\documentclass[12pt,twoside,reqno]{amsart}
\usepackage{amsmath}
\usepackage{amsfonts}
\usepackage{amssymb}
\usepackage{color}
\usepackage{mathrsfs}
\usepackage{cite}
\usepackage{graphicx}
\usepackage{times}
\usepackage[usenames,dvipsnames]{xcolor}
\newcommand{\R}{\mathbb R}

\newtheorem{theorem}{Theorem}[section]

\newtheorem{cor}[theorem]{Corollary}
\newtheorem{prop}[theorem]{Proposition}
\newtheorem{rems}[theorem]{Remarks}
\newtheorem{rem}[theorem]{Remark}

\numberwithin{equation}{section}
\numberwithin{theorem}{section}
\newcommand{\ve}{\varepsilon}
\newcommand{\rd}{\mathrm{d}}
\begin{document}
\title{Some singular equations modeling MEMS}
\thanks{Partially supported by the French-German PROCOPE project 30718ZG}

\author{Philippe Lauren\c{c}ot}
\address{Institut de Math\'ematiques de Toulouse, UMR~5219, Universit\'e de Toulouse, CNRS \\ F--31062 Toulouse Cedex 9, France}
\email{laurenco@math.univ-toulouse.fr}

\author{Christoph Walker}
\address{Leibniz Universit\"at Hannover, Institut f\" ur Angewandte Mathematik, Welfengarten 1, D--30167 Hannover, Germany}
\email{walker@ifam.uni-hannover.de}

\keywords{Microelectromechanical system - Free boundary problem - Nonlocal nonlinearity - Finite time singularity - Well-posedness - Beam equation - Wave equation}
\subjclass{35Q74 - 35R35 - 35M33 - 35K91 - 35B44 - 35B65}

\date{\today}

%%%%%%%%%%%%%%%%%%%%%%%
%%%%%%%%%%%%%%%%%%%%%%%
\begin{abstract}
In the past fifteen years mathematical models for microelectromechanical systems (MEMS) have been the subject of several studies, in particular due to the interesting qualitative properties they feature. Still most research is devoted to an illustrative, but simplified model which is deduced from a more complex model when the aspect ratio of the device vanishes, the so-called vanishing (or small) aspect ratio model. The analysis of the aforementioned complex model involving a moving boundary has started only recently and an outlook of the results obtained so far in this direction is provided in this survey.
\end{abstract}
%%%%%%%%%%%%%%%%%%%%%%%
%%%%%%%%%%%%%%%%%%%%%%%

\maketitle

%
%     HEADLINES
%
\pagestyle{myheadings}
\markboth{\sc{Ph. Lauren\c cot, Ch. Walker}}{\sc{MEMS models with and without free boundary}}

\tableofcontents
%

%%%%%%%%%%%%%%%%%%%%%%%
%%%%%%%%%%%%%%%%%%%%%%%
\section{Introduction}\label{secint}
%%%%%%%%%%%%%%%%%%%%%%%
%%%%%%%%%%%%%%%%%%%%%%%

Feynman anticipated and popularized the growing need for micro- and nanostructures with his famous lecture entitled ``There is plenty of room at the bottom'' \cite{Fey92} on the occasion of the American Physical Society's annual meeting  in 1959. Less than a decade later, Nathanson et al. \cite{NNWD67} produced in 1967 the resonant gate transistor: the first batch-fabricated {\it microelectromechanical system (MEMS)} converting an input electrical signal into a mechanical force by utilizing electrostatic attraction. Since then, integrated devices on the scale of $1-100 \mu m$, combining mechanical and electrical components, have become a well-established technology for everyday electronic products with applications in various domains, including automotive domain (e.g. accelerometers in airbag deployment systems, fuel or air pressure sensors, rollover detection, headlight leveling), commercial domain (e.g. inertia sensors in smartphones, inkjet printer heads), biotechnology (e.g. sensors for infusion pumps), industrial domain (e.g. optical switches, microvalves, micropumps, capacitors, gyroscopes), and many more \cite{FMCCS05, You11, PeB03, Jaz12, Jaz12b}. With the increasing demand for miniaturized sensors and actuators of sophisticated functionalities, MEMS is regarded as one of the most promising technologies of this century. Needless to say that there are many challenges associated with this advancing miniaturization, both from a technological and theoretical point of view.

\medskip

Electrostatically actuated MEMS devices suffer from an ubiquitous {\it pull-in instability} which limits their effectiveness. This phenomenon has been identified and described already in the seminal work of Nathanson et al. \cite{NNWD67} for the simple situation of a lumped mass and spring system with two parallel plates at different potentials as depicted in Figure~\ref{MassSpring}. 
%%%%%%%%%%%%%%%%%%%%%%%
\begin{figure}[h]
\centering\includegraphics[scale=.45]{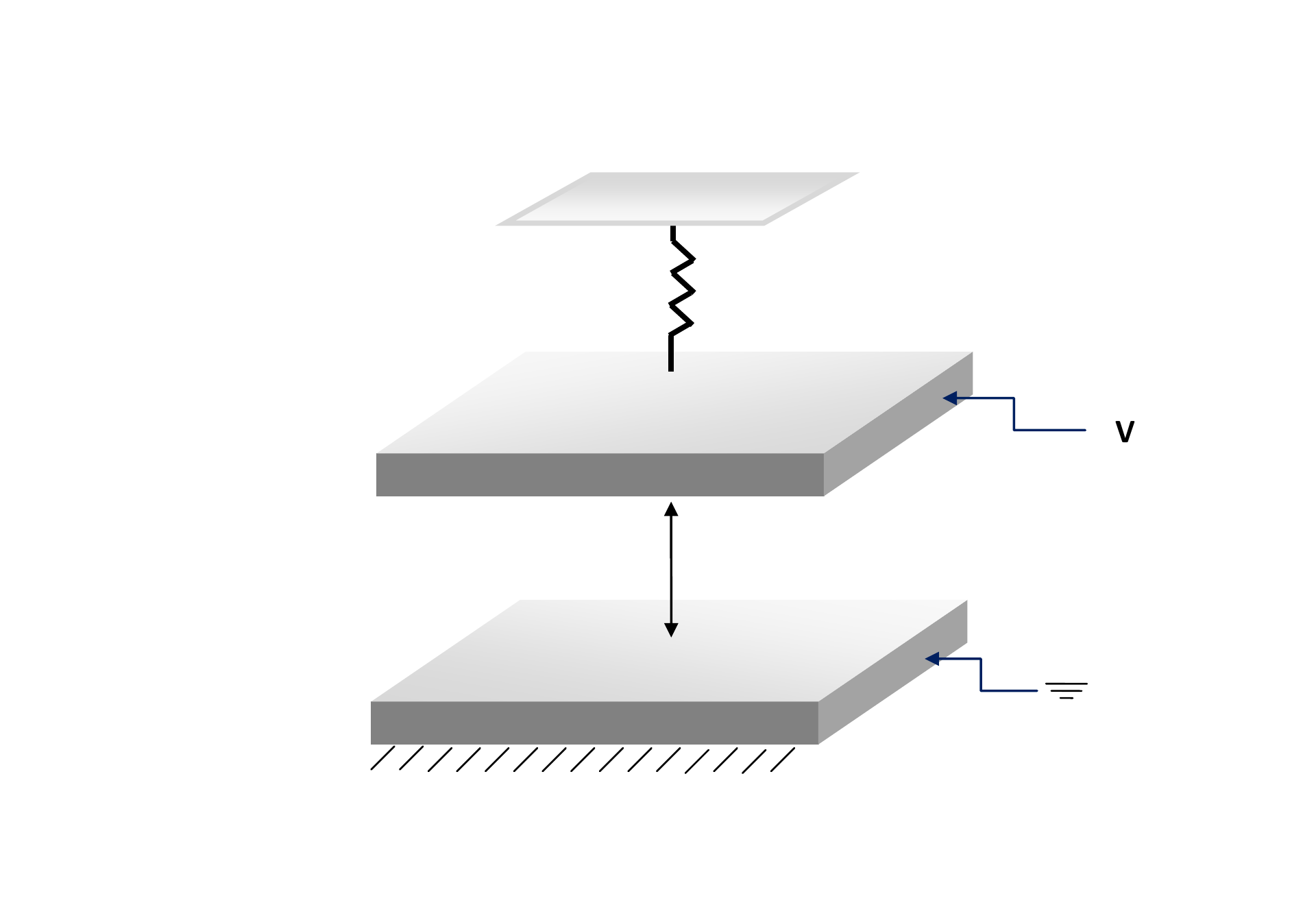}
\caption{\small Schematic drawing of a spring mass system.}\label{MassSpring}
\end{figure}
%%%%%%%%%%%%%%%%%%%%%%%
In this instability, when applied voltages are increased beyond a certain critical value, the top plate ``pulls in'' in the sense that the plates do no longer remain separate, restricting the range of stable operating of the device. Whether or not the \textsl{touchdown} of the upper plate on the lower one is a sought-for effect, key issues are not only to determine the range of applied voltages for which this phenomenon takes place but also the state of the device when it does not (for instance, it might be of interest to optimize the gap between the two plates). For the design of reliable high-performance devices, a detailed understanding of stable operating regimes and the occurrence of pull-in instabilities is thus of utmost importance. 

Even in greatly simplified and idealized physical situations, the governing equations are mathematically challenging. So not only the above mentioned applications are of true interest, but also the mathematical investigation of the equations which involves various tools from many different fields from nonlinear analysis. The focus of this article is on a mathematical model for an idealized electrostatically actuated device in its ``original'' form of a free boundary problem \cite[Section~7.4]{PeB03}. Maybe due to the many mathematical difficulties that come along with it, this free boundary problem was merely the source of simplified models and the research has focused so far mostly on a  version thereof where the ratio height/length of the device is taken to vanish, thereby allowing one to compute the moving domain and to reduce the original free boundary problem to a single partial differential equation. While we only provide a short overview on this simplified model, 
we shall introduce the reader to some recent advances in the mathematical analysis of the free boundary problem.

%%%%%%%%%%%%%%%%%%%%%%%
%%%%%%%%%%%%%%%%%%%%%%%
\subsection{The Free Boundary Problem for an Idealized MEMS}
%%%%%%%%%%%%%%%%%%%%%%%
%%%%%%%%%%%%%%%%%%%%%%%

An idealized electrostatically actuated MEMS device which features the above mentioned instability is schematically drawn in Figure~\ref{MEMSfig}. 
%%%%%%%%%%%%%%%%%%%%%%%
\begin{figure}[h]
\centering\includegraphics[scale=.45]{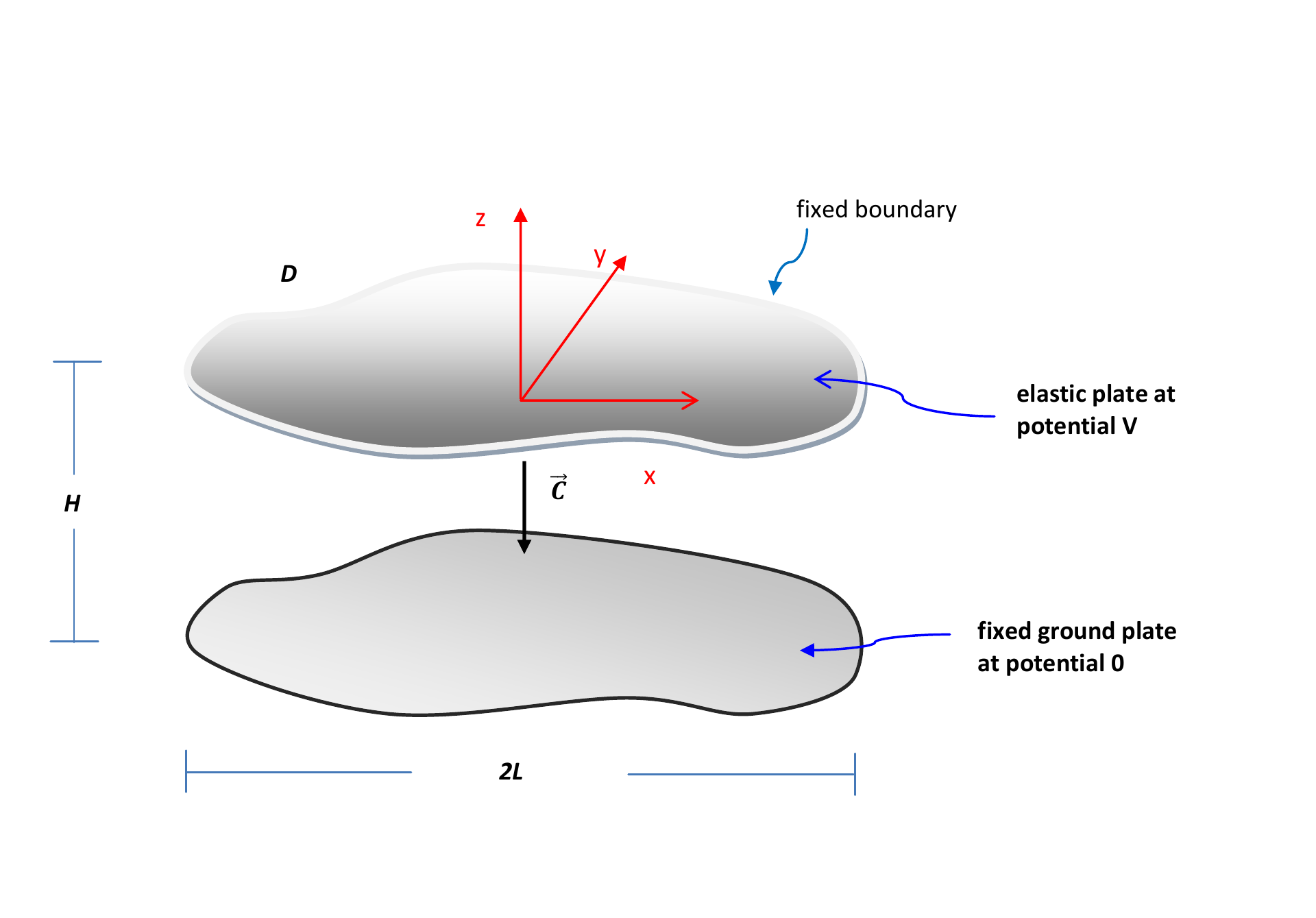}
\caption{\small Schematic diagram of an idealized electrostatic MEMS device.}\label{MEMSfig}
\end{figure}
%%%%%%%%%%%%%%%%%%%%%%%
It consists of a rigid conducting ground plate above which an elastic plate (or membrane), coated with a thin layer of dielectric material and clamped on its boundary, is suspended. Holding the ground plate and the deformable plate at different electric potentials induces a Coulomb force across the device resulting in a deformation of the elastic plate and thus in a change in geometry, while electrostatic energy is converted into mechanical energy. There is thus a competition between attractive electrostatic forces and restoring elastic forces acting on the upper plate. After a suitable rescaling, the elastic plate's position at rest is a two-dimensional domain $D$ located at $z=0$ while the fixed ground plate of the same shape is positioned at $z=-1$. The electrostatic potential on the ground plate is assumed to be zero while it takes a positive value $V>0$ on the elastic plate. The state of the device is fully described by the deformation $u=u(t,x,y)$ of the elastic plate in $z$-direction and the (dimensionless) electrostatic potential $\psi=\psi(t,x,y,z)$ between the plates. Introducing the aspect ratio $\ve$ of the device (i.e. the ratio height/length), the electrostatic potential satisfies the rescaled Laplace equation
\begin{equation}\label{psi}
\varepsilon^2\partial_x^2\psi + \varepsilon^2\partial_y^2\psi +\partial_z^2\psi =0\ ,\quad (x,y,z)\in \Omega(u(t))\ ,\quad t>0\ ,
\end{equation}
in the time-varying cylinder
$$
\Omega(u(t)) := \left\{ (x,y,z)\in D\times (-1,\infty)\ :\ -1 < z < u(t,x,y) \right\}
$$
between the rigid ground plate and the deflected elastic plate. Scaling the potential difference to $V=1$, the boundary conditions for $\psi$ are taken to be
\begin{equation}\label{bcpsi}
\psi(t,x,y,z)=\frac{1+z}{1+u(t,x,y)}\ ,\quad (x,y,z)\in  \partial\Omega(u(t))\, , \quad t>0\ ,
\end{equation}
that is, $\psi=1$ on the elastic plate and $\psi=0$ on the ground plate. The dynamics of the elastic plate deformation is governed by a damped wave/plate equation
\begin{equation}\label{u}
\begin{split}
\gamma^2\partial_t^2 u &+\partial_t u +\beta\Delta^2 u-  \tau \Delta u \\
& =-\lambda\, \left( \ve^2 |\nabla ' \psi(t,x,y,u(t,x,y))|^2 + |\partial_z \psi(t,x,y,u(t,x,y))|^2 \right) 
\end{split}
\end{equation}
for time $t>0$ and position $(x,y)\in D$ with clamped boundary conditions
\begin{equation}\label{bcu}
u=\beta\partial_\nu u=0\ \text{ on }\ \partial D\ , \quad t>0\, ,
\end{equation}
see, e.g., \cite{PeB03, Jaz12, FMCCS05, You11} and Section~\ref{Sec2.1}. In \eqref{u} we use the notation $\nabla ':= (\partial_x,\partial_y)$ for the gradient with respect to the horizontal directions. A commonly used initial state is that the elastic plate is at rest, that is,
\begin{equation}\label{icu}
u(0,x,y)=\gamma\partial_tu(0,x,y)=0\ ,\qquad (x,y)\in D\ .
\end{equation}
The fourth-order term $\beta \Delta^2 u$ in equation \eqref{u} reflects plate bending while the linear second-order term $\tau \Delta u$ with $\tau\ge 0$  accounts for external stretching. Taking $\beta>0$ (with $\tau\ge 0$)  reflects the case of an elastic plate while $\beta=0$ (with $\tau>0$) rather refers to an elastic membrane.
The right-hand side of \eqref{u} is due to  the electrostatic forces exerted on the elastic plate. The parameter $\lambda>0$, being proportional to the square of the applied voltage difference and the device's aspect ratio~$\varepsilon >0$ are key quantities in the model and its analysis. In the sequel we focus on the situation of an elastic plate clamped at its boundary as modeled in \eqref{bcu} though other boundary conditions are possible. For instance, assuming that the domain $D$ is regular enough and denoting the curvature of $\partial D$ (taken to be positive for convex domains) by $\kappa$, the hinged boundary conditions (also referred to as Steklov boundary conditions) 
\begin{equation}\label{bcuSteklov}
u= \beta\Delta u-\beta(1-\sigma)\kappa \partial_\nu u=0 \text{ on }\ \partial D\,, \quad t>0\ ,
\end{equation}
are derived from the description of the deformation of a device which is ideally hinged along all its boundary so that it is free to rotate and cannot support a bending moment \cite{Vil97, GS08}, the parameter $\sigma$ being the Poisson ratio which in general ranges in $[0,1/2)$. A variation hereof, often met in the literature, are the pinned boundary conditions
\begin{equation}\label{bcubcuPinned}
u=\beta\Delta u=0\ \text{ on }\ \partial D\,, \quad t>0\ .
\end{equation}

From a mathematical point of view, the aforementioned pull-in phenomenon occurring when the elastic plate contacts the ground plate corresponds to the deflection $u$ attaining the value $-1$ somewhere in $D$  at some time $t>0$ and is also referred to as the \textit{touchdown} phenomenon. Equations~\eqref{psi}-\eqref{bcpsi} are then no longer meaningful since the set $\Omega(u(t))$ becomes disconnected and solutions cease to exist. Moreover, if $(x_0,y_0)$ denotes a touchdown point in $D$, the vertical derivative $\partial_z\psi(t,x_0,y_0,u(t,x_0,y_0))$ appearing on the right-hand side of \eqref{u} develops a singularity as $\psi=1$ along $z=u(t,x,y)$ while $\psi=0$ along $z=-1$. In some sense, the parameter $\lambda$ tunes this singularity and consequently plays a crucial r\^{o}le in the pull-in instability which is closely related to the existence of stationary (i.e. time-independent) solutions and the question of global existence of solutions to~\eqref{psi}-\eqref{icu}.

The elliptic equation \eqref{psi} is a free boundary problem posed in the non-smooth cylinder $\Omega(u(t))$ which changes its geometry as its upper surface -- that is, the elastic plate being described at instant $t$ by $u(t)$ -- evolves in time. The inherently nonlinear character of this free boundary problem induces another mathematical challenge to the analysis of \eqref{psi}-\eqref{icu} in the form of an intricate coupling of the two unknowns $\psi$ and $u$\,: Any solution $\psi$ to \eqref{psi}-\eqref{bcpsi} and accordingly the right-hand side of the wave equation \eqref{u}  are {\it nonlocal} and nonlinear functions of $u$ and likely to enjoy restricted regularity. It is thus certainly not obvious how the function $\psi$ reacts on changes with respect to $u$. 

In this connection, let us point out that the system \eqref{psi}-\eqref{bcu} has a gradient structure \cite{BeP11, BrP11, BSSDP13, FMCCS05, LW14a}. More precisely, there is an energy functional $\mathcal{E}$ defined for suitable deflections $u$ such that stationary solutions to \eqref{psi}-\eqref{bcu} correspond to critical points of $\mathcal{E}$ while it decays along any trajectory of the evolution problem. However, as we shall see below, the use of this variational structure is delicate due to the non-coercivity of the functional. In fact, the energy $\mathcal{E}(u)$ diverges to $-\infty$ if the minimum of $u$ approaches $-1$. Nonetheless, this structure provides a simple way of deriving the model (Section~\ref{S2}) and is also helpful in excluding the occurrence of some singularities.

%%%%%%%%%%%%%%%%%%%%%%%
%%%%%%%%%%%%%%%%%%%%%%%
\subsection{The Vanishing Aspect Ratio Model for an Idealized MEMS}
%%%%%%%%%%%%%%%%%%%%%%%
%%%%%%%%%%%%%%%%%%%%%%%

The geometrical difficulties in the governing equations can be hidden when the aspect ratio~$\ve$ is negligibly small. Indeed, setting formally $\ve=0$ in \eqref{psi}, we readily infer from \eqref{psi}-\eqref{bcpsi} that the electrostatic potential $\psi_0$ is given by the explicit formula
\begin{equation}\label{psiSG}
\psi_0(t,x,y,z)=\frac{1+z}{1+u_0(t,x,y)}
\end{equation}
as a function of the yet to be determined deflection $u_0$ according to
\begin{equation}\label{uSG}
\begin{split}
\gamma^2\partial_t^2 u_0+\partial_t u_0 +\beta\Delta^2 u_0  - \tau \Delta u_0 = - \frac{\lambda}{(1+u_0)^2}
\end{split}
\end{equation}
for $t>0$, $(x,y)\in D$, supplemented with the boundary condition \eqref{bcu} and initial conditions \eqref{icu}. The original free boundary problem is thus explicitly solved and equation~\eqref{u} is reduced to a single semilinear equation~\eqref{uSG} known as the {\it vanishing aspect ratio} equation. The nonlinear reaction term on its right-hand side is no longer nonlocal but still singular when $u_0=-1$, which corresponds to the touchdown phenomenon already alluded to. This singularity is also referred to as \textit{quenching} in the literature concerning parabolic equations with singular right-hand sides, see \cite{FHQ92, Lev89, Lev92} and the references therein. In addition, from the explicit form of the nonlinearity on the right-hand side of \eqref{uSG} utmost favorable properties can be directly read off, important and well exploited in the investigation of qualitative aspects of solutions: the reaction term of equation~\eqref{uSG} is monotone, concave, and of zero order. These properties are not known and by no means obvious for the equation~\eqref{u} involving the nonlocal reaction term on its right-hand side. Still, the singular equation~\eqref{uSG} depending on the parameter~$\lambda$ presents various mathematical challenges and exhibits many interesting phenomena which is the reason why the vast literature dedicated to vanishing aspect ratio models for MEMS is continuously growing.

%%%%%%%%%%%%%%%%%%%%%%%
%%%%%%%%%%%%%%%%%%%%%%%
\subsection{A Bird's Eye View on MEMS Dynamics}\label{Sec1.3}
%%%%%%%%%%%%%%%%%%%%%%%
%%%%%%%%%%%%%%%%%%%%%%%

As we shall see more precisely later on, the qualitative behavior of solutions to both \eqref{u} and \eqref{uSG} significantly depend on whether or not some of the coefficients $\beta$, $\gamma$, and $\tau$ are assumed to be zero as well as on the shape and the dimension of the plate $D$. In particular, the order of the partial differential equation (fourth-order when $\beta>0$ versus second-order when $\beta=0$) and its  character  (hyperbolic when $\gamma>0$ versus parabolic when $\gamma=0$) play important r\^oles.

As already mentioned a common peculiarity of the dynamics of MEMS is the pull-in instability which is widely observed in real world applications and extensively discussed in engineering literature (see, e.g., \cite{You11, FMCCS05} for references).
Mathematically it manifests in different ways: 

\begin{itemize}

\item[$\star$] In the stationary regime it is (presumably) revealed as a sharp stationary threshold $\lambda_\ve^{stat}>0$ (depending on $\varepsilon$) for the voltage parameter $\lambda$ above which there is no stationary solution while below there is at least one, the exact number of steady-state solutions for a specific value of $\lambda \in (0,\lambda_\ve^{stat}]$ depending heavily on the space dimension. This threshold value of the voltage parameter is often referred to as the (stationary) {\it pull-in voltage} in the literature.

\item[$\star$] For the evolution problem the pull-in instability corresponds to the occurrence of a finite time singularity, the {\it touchdown phenomenon}, in which the elastic plate, initially at rest, collapses onto the ground plate or, equivalently, the plate deflection $u$ reaches the value $-1$. Similarly as in the stationary case, it is expected that a dynamic threshold value $\lambda_{\ve,\gamma}^{dyn}>0$ (depending on $\varepsilon$ and $\gamma$) exists for the dynamical regime such that touchdown takes place when $\lambda>\lambda_{\ve,\gamma}^{dyn}$ while the deflection $u$ stays above $-1$ for all times when $\lambda<\lambda_{\ve,\gamma}^{dyn}$. It is worth pointing out that, as observed in \cite{KLNT15}, this threshold value is likely to depend on the initial deformation of the elastic plate. We emphasize that throughout this paper ``pull-in'' in the dynamic case always refers to a touchdown of the elastic plate when it is initially at rest as in \eqref{icu}.
\end{itemize}

Neither the existence of {\it sharp} pull-in voltages $\lambda_\ve^{stat}$ and $\lambda_{\ve,\gamma}^{dyn}$ nor the relation between them are fully established in all cases, not even for the vanishing aspect ratio model \eqref{uSG} where $\ve=0$. In fact, for the fourth-order case  $\beta>0$ with clamped boundary conditions \eqref{bcu},  the only available result appears to be that $\lambda_ {0,\gamma}^{dyn}\le \lambda_0^{stat}$ as shown in \cite{LW14b, LiL12}. Still when $\ve=0$, more can be said regarding the second-order setting with $\beta=0$ which we describe now, see also Section~\ref{Sec4.2}. It is known \cite{GhG08b} that
\begin{equation}\label{lambda*}
\lambda_0^{stat}=\lambda_{0,0}^{dyn}\,,
\end{equation}
while numerical results \cite{Flo14} and experimental investigations \cite{SDSP11} indicate that
\begin{equation}\label{lambda**}
\lambda_{0,\gamma}^{dyn} < \lambda_0^{stat}\,,\quad \gamma>0\,.
\end{equation} 
This implies that the hyperbolic evolution can drive the device to a touchdown even though there is a stationary configuration. Thus, precise details and the exact values of pull-in voltages  are of utmost importance from the viewpoint of applications.

\medskip

We also point out that the order of the differential equation impacts not only on the structure of the stationary solutions but also on the dynamics of the touchdown behavior, at least for the vanishing aspect ratio model~\eqref{uSG}. Indeed, on the one hand, the fourth-order ($\beta>0$) stationary problem in a one- or two-dimensional domain $D$ is in many respects qualitatively the same as that of the second-order equation ($\beta=0$) in a one-dimensional interval $D$ (see \cite{EGG10} and Section~\ref{Sec2.2}). On the other hand, there is numerical evidence that touchdown occurs in one space dimension at a single point when $\beta=0$ and $\lambda>\lambda_{0}^{stat}=\lambda_{0,0}^{dyn}$, while it takes place at two different points when $\beta>0$ and $\lambda\gg\lambda_{0,0}^{dyn}$ \cite{LiL12}.

\medskip

A similar behavior as observed in the vanishing aspect ratio equation is expected for the free boundary problem  \eqref{psi}-\eqref{icu} with $\ve>0$, that is, that a stationary and dynamical pull-in voltage values $\lambda_\ve^{stat}$ and~$\lambda_{\ve,\gamma}^{dyn}$ exist which may be different. Numerical evidence \cite{FlSxx} implies that the values for the pull-in voltages indeed depend on both $\ve$ and $\gamma$ and, furthermore,  on the hyperbolic or parabolic character since
\begin{equation}\label{lambdafb}
 \lambda_{\ve,\gamma}^{dyn}<\lambda_{\ve,0}^{dyn}= \lambda_\ve^{stat}\ ,\quad \ve>0\,,\quad \gamma> 0\,.
\end{equation} 

\subsection{Outline} In the following, we shall give more details regarding analytical results on the free boundary problem~\eqref{psi}-\eqref{icu}. In Section~\ref{S2} we formally derive the equations as Euler-Lagrange equations of the underlying energy balance within linear (Section~\ref{Sec2.1}) and nonlinear (Section~\ref{Sec2.3}) elasticity theory. We also present the equations in reduced dimensions (Section~\ref{Sec2.2}) --~which will be the main focus of this article~-- and briefly discuss various extensions of the model (Section~\ref{Sec2.4}), which, however, will not be pursued further in this article. In Section~\ref{S3} we collect some useful tools needed for the analysis of~\eqref{psi}-\eqref{icu}. We first investigate the subproblem \eqref{psi}-\eqref{bcpsi} for the electrostatic potential and show how to formulate the free boundary problem as a single equation for the elastic plate deflection (Section~\ref{Sec3.1}). We then discuss the rigorous proof that the free boundary problem has indeed a variational structure (Section~\ref{Sec3.2}) and finally present some results related to the fourth-order maximum principle (Section~\ref{Sec3.3}), which -- when available -- is a central tool for the analysis. Section~\ref{S4} is dedicated to existence and non-existence results for stationary solutions and the related pull-in instability (Section~\ref{Sec4.1}). We also summarize some of the research on stationary solutions to the vanishing aspect ratio model (Section~\ref{Sec4.2}). The evolution problem is considered in Section~\ref{S5} which includes local (Section~\ref{Sec5.1}) and \mbox{(non-)} global existence results (Section~\ref{Sec5.2}) with regard to the pull-in phenomenon. This section also contains a brief discussion on the damping dominated limit (Section~\ref{Sec5.3}) and the vanishing aspect ratio limit (Section~\ref{Sec5.4}). In Section~\ref{S6} we give a short summary on the (original) three-dimensional case. Finally, in Section~\ref{S7} we mention some open problems that could give raise to further investigations.

%%%%%%%%%%%%%%%%%%%%%%%
%%%%%%%%%%%%%%%%%%%%%%%
\section{Derivation of the Governing Equations}\label{S2}
%%%%%%%%%%%%%%%%%%%%%%%
%%%%%%%%%%%%%%%%%%%%%%%

The equations governing the dynamics of the elastic plate in an electrostatically actuated MEMS may be derived from its energy balance after a suitable rescaling as Euler-Lagrange equations \cite{BeP11, BrP11, BSSDP13, FMCCS05, LW14c, PeB03}. We shall first derive the equations presented in the introduction using linear elasticity theory and then introduce variants thereof under different modeling assumptions.

%%%%%%%%%%%%%%%%%%%%%%%
%%%%%%%%%%%%%%%%%%%%%%%
\subsection{Linear Elasticity Model}\label{Sec2.1}
%%%%%%%%%%%%%%%%%%%%%%%
%%%%%%%%%%%%%%%%%%%%%%%

In Figure~\ref{MEMSfig}, an idealized MEMS device is schematically drawn. We assume that the elastic plate, held at potential $V$ and clamped along its boundary, is of shape $\hat D \subset \R^2$ with characteristic dimension $2 L$ and rest position at $\hat z=0$. The grounded fixed plate is also of shape $\hat D$ and located at $\hat z=-H$. 
We denote the vertical deflection\footnote{See e.g. \cite{NYAR05} for a more sophisticated model including non-vertical large deflections based on the von K\'arm\'an equations.} of the elastic plate at position $(\hat{x},\hat y)\in\hat D $ and time $\hat{t}$ by $\hat{u}=\hat{u}(\hat{t},\hat{x},\hat y) >-H$ and the electrostatic potential between the plates at position $(\hat{x},\hat y,\hat{z})$ and time $\hat{t}$ by $\hat{\psi}=\hat{\psi}(\hat{t},\hat{x} ,\hat y ,\hat{z})$. We suppress the time variable $\hat{t}$ for the moment.
The electrostatic potential $\hat{\psi}$ is harmonic in the region
$$
\hat{\Omega}(\hat{u}):=\left\{(\hat{x},\hat y,\hat{z})\,;\, (\hat x,\hat y)\in \hat D\,,\, -H<\hat{z}<\hat{u}(\hat{x},\hat y)\right\}
$$
between the ground plate and the elastic plate, that is,
\begin{equation*}\label{psihat}
\Delta\hat{\psi}=0 \quad\text{in}\quad \hat{\Omega}(\hat{u})
\end{equation*}
and with given values on the plates
\begin{equation*}\label{psihatbc}
\hat{\psi}(\hat{x},\hat y, -H)=0\ ,\quad \hat{\psi}(\hat{x},\hat y,\hat{u}(\hat{x},\hat y))= V\ ,
 \qquad (\hat{x},\hat y)\in\hat D\ ,
\end{equation*}
extended continuously on the vertical sides of $\hat{\Omega}(\hat{u})$ by
\begin{equation*}\label{psihatbcside}
\hat{\psi}(\hat{x},\hat{y},\hat{z}) = V(1+\hat{z})\ , \quad (\hat{x},\hat{y})\in \partial\hat{D}\ .
\end{equation*}
This choice is motivated by formula  \eqref{psiSG} for $\psi$ in the vanishing aspect ratio limit.
The {\it total energy} generated by the deformation of the elastic plate is \cite{BeP11, BrP11, BSSDP13, FMCCS05, PeB03}
$$
\hat{\mathcal{E}}(\hat{u})=\hat{\mathcal{E}}_b(\hat{u})+\hat{\mathcal{E}}_s(\hat{u})+\hat{\mathcal{E}}_e(\hat{u})\ . 
$$ 
It involves the \textit{bending energy}~$\hat{\mathcal{E}}_b$, for simplicity taking into account only the linearized curvature of the elastic plate, that is,
\begin{equation*} 
\hat{\mathcal{E}}_b(\hat{u})=\frac{Ybh^3}{24}\int_{\hat D}  \left\vert\Delta \hat u(\hat{x},\hat y)\right\vert^2\,\rd (\hat{x},\hat y)\ ,
\end{equation*}
where $Y$ is Young's modulus, $h$ is the height of the plate and $b$ its width, the latter being assumed to be of the same order as the plate's length $2L$. The {\it stretching energy} $\hat{\mathcal{E}}_s$ retains the linearized contribution due to the axial tension force with coefficient $T$ together with a self-stretching force coming from elongation of the plate in the case of (moderately) large oscillations. It is of the form \cite{FMCCS05}
$$
\hat{\mathcal{E}}_s(\hat{u})=\frac{T}{2}\int_{\hat D}  \left\vert \nabla\hat u(\hat{x},\hat y)\right\vert^2\, \rd (\hat{x},\hat y)\ +\ \frac{bhY}{8L}\left(\int_{\hat D}  \left\vert \nabla\hat u(\hat{x},\hat y)\right\vert^2\, \rd (\hat{x},\hat y)\right)^2\ .
$$
The {\it electrostatic energy} $\hat{\mathcal{E}}_e$  is given by
$$
\hat{\mathcal{E}}_e(\hat{u})=-\frac{\epsilon_0}{2}\int_{ \hat\Omega(\hat u)}\vert\nabla\hat{\psi}(\hat{x},\hat y,\hat z)\vert^2\,\rd (\hat{x},\hat y,\hat z)
$$
with $\epsilon_0$ being the permittivity of free space. Note that $\hat{\psi}$ is the maximizer of the (negative) Dirichlet energy $-(\epsilon_0/2) \|\nabla\hat{\xi}\|_{L_2(\hat{\Omega}(\hat{u}))}^2$ among functions $\hat{\xi}\in H^1(\hat{\Omega}(\hat{u}))$ satisfying the prescribed boundary conditions on $\partial\hat{\Omega}(\hat{u})$. Introducing dimensionless variables
$$
x=\frac{\hat{x}}{L}\ ,\quad y=\frac{\hat{y}}{L}\ ,\quad z=\frac{\hat{z}}{H}\ ,\quad u=\frac{\hat{u}}{H}\ ,\quad \psi = \frac{\hat{\psi}}{V}\ ,
$$
and denoting  the aspect ratio of the device by $\ve=H/L$, the total energy in these variables is
\begin{equation*}\label{E}
\begin{split}
\mathcal{E}(u)&= \frac{Ybh^3\ve^2}{24}\int_{D}  \left\vert\Delta u(x, y)\right\vert^2\,\rd (x, y)
+ \frac{TL^2\ve^2}{2}\int_{D}  \left\vert \nabla u(x, y)\right\vert^2\, \rd (x, y) \\
&\quad + \frac{bhYL^3\ve^4}{8}\left(\int_{ D}  \left\vert \nabla u(x, y)\right\vert^2\, \rd (x, y)\right)^2\\
&\quad -\frac{\epsilon_0V^2L}{2\ve}\int_{\Omega(u)}\big(\ve^2\vert\partial_{x}\psi(x,y,z)\vert^2+\ve^2\vert\partial_{y}\psi(x,y,z)\vert^2+\vert\partial_z\psi(x,y,z)\vert^2\big)\,\rd (x,y,z)
\end{split}
\end{equation*} 
with $D:=\{ (x,y)\,;\, (Lx,Ly)\in \hat D\}$ and
$$
\Omega(u) := \left\{ (x,y,z)\in D\times (-1,\infty)\ :\ -1 < z < u(x,y) \right\}\,.
$$
The equilibrium configurations of the device are the critical points of the total energy, i.e. satisfy $\partial_u \mathcal{E}(u)=0$. The computation of this Fr\'echet derivative with respect to $u$ is quite involved since $\psi$ solves the rescaled equation
\begin{equation*}\label{psi1}
\varepsilon^2\partial_x^2\psi + \varepsilon^2\partial_y^2\psi +\partial_z^2\psi =0\ ,\quad (x,y,z)\in \Omega(u)\ ,
\end{equation*}
in the $u$-dependent domain $\Omega(u)$ along with the boundary condition 
\begin{equation*}\label{bcpsi1}
\psi(x,y,z)=\frac{1+z}{1+u(x,y)}\ ,\quad (x,y,z)\in  \partial\Omega(u)\,  ,
\end{equation*}
and hence $\psi=\psi_u$ depends in a nonlocal way on the deformation $u$. Interpreting the derivative of the electrostatic energy as the shape derivative of the Dirichlet integral of $\psi=\psi_u$ and using shape optimization arguments \cite{HeP05},  the Euler-Lagrange equation is (see Section~\ref{Sec3.2} and \cite{LW14a})
\begin{equation*}\label{EL}
\begin{split}
0&=
-\ve^2\,\beta \,\Delta^2 u+\ve^2\left(\tau+a\|\nabla u\|_2^2\right)\Delta u\\
&\quad -\ve^2\lambda \big(\ve^2\vert\nabla'\psi(x,y,u(x,y))\vert^2+\vert\partial_z\psi(x,y,u(x,y))\vert^2\big)
\end{split}
\end{equation*}
for $(x,y)\in D$, where $\nabla '=(\partial_x,\partial_y)$ and
$$
\beta:=\frac{Ybh^3}{12}\,,\quad \tau:=TL^2\,,\quad a=a(\ve):=\frac{bh Y L^3 \ve^2}{2}\,, \quad \lambda=\lambda(\ve):=\frac{\epsilon_0V^2 L}{2\ve^3}\ .
$$
The dynamics of the time-dependent deflection \mbox{$u=u(\hat{t},x,y)$} is derived by means of Newton's second law. Letting  $\rho$ and $\delta $ denote the mass density per unit volume of the elastic plate respectively its thickness, the sum over all forces equals  $\rho \delta  \partial_{\hat{t}}^2 u$. The elastic and electrostatic forces are combined with a damping force of the form $-r\partial_{\hat{t}} u$ being linearly proportional to the velocity. Scaling time based on the strength of damping according to $t=\hat{t}\ve^2/r$ and setting \mbox{$\gamma:=\ve\sqrt{\rho \delta }/r$}, we derive for the dimensionless deflection $u$ the evolution equation
\begin{equation*}\label{evol}
\begin{split}
\gamma^2\, \partial_{t}^2 u +\partial_{t} u\, & =\,-\beta \,\Delta^2 u+\left(\tau+a\|\nabla u\|_2^2\right)\Delta u\\
&\quad 
-\lambda \big(\ve^2\vert\nabla '\psi(t,x,y,u(t,x,y))\vert^2+\vert\partial_z\psi(t,x,y,u(t,x,y))\vert^2\big)
\end{split}
\end{equation*}
for $t>0$ and $(x,y)\in D$, which is \eqref{u} when $a=0$, that is, when self-stretching forces due to elongation of the plate are neglected. 

%%%%%%%%%%%%%%%%%%%%%%%
\begin{rem}
An alternative, though formal, way to obtain the electrostatic force exerted on the elastic plate, which avoids the use of shape derivatives, is to treat the device as a parallel plate capacitor \cite{PeB03, Pel01a, EGG10, LiY07} so that classical electrostatics can be applied, see e.g. \cite{Jac62}.
\end{rem}
%%%%%%%%%%%%%%%%%%%%%%%

%%%%%%%%%%%%%%%%%%%%%%%
%%%%%%%%%%%%%%%%%%%%%%%
\subsection{A One-Dimensional Linear Elasticity Model}\label{Sec2.2}
%%%%%%%%%%%%%%%%%%%%%%%
%%%%%%%%%%%%%%%%%%%%%%%

A common building block in MEMS devices are rectangular plates, i.e., after rescaling,
$$
D=(-1,1)\times (-b/L,b/L)\subset\R^2\,,
$$ 
where the elastic plate is only held fixed along the edges in $y$-direction while the edges in $x$-direction are free. This structure is depicted in Figure~\ref{MEMS1d} and yields a one-dimensional equation for the plate deformation when assuming homogeneity in $y$-direction. 
%%%%%%%%%%%%%%%%%%%%%%%
\begin{figure}[h]
\centering\includegraphics[scale=.45]{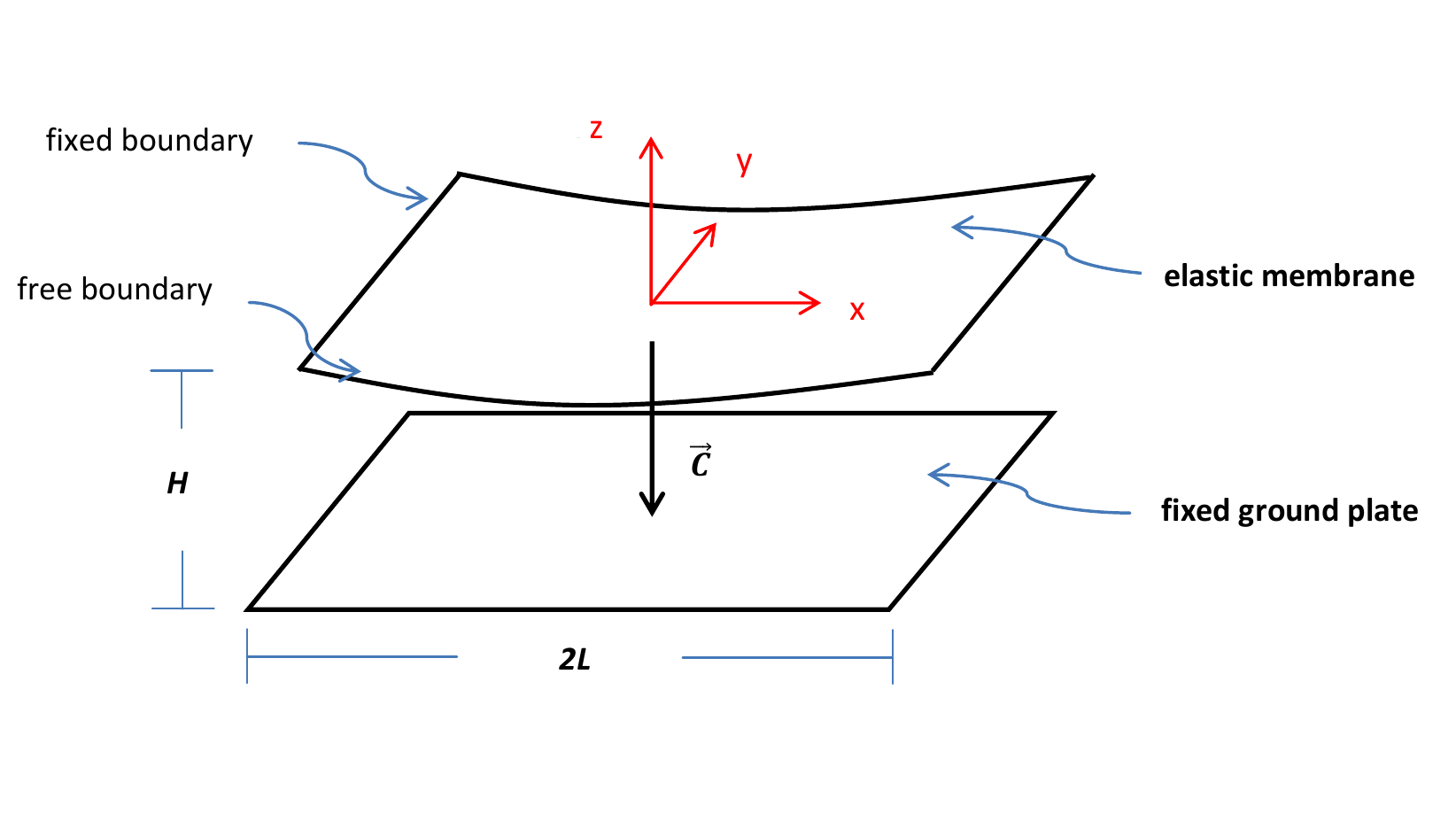}
\caption{\small Schematic diagram of an idealized electrostatic MEMS device of rectangular shape yielding a one-dimensional deformation.}\label{MEMS1d}
\end{figure}
%%%%%%%%%%%%%%%%%%%%%%%
In this case, $u=u(t,x)$ and $\psi=\psi(t,x,z)$ are independent of the variable $y$ so that the equations \eqref{psi}-\eqref{icu} reduce to
\begin{align}
\varepsilon^2\partial_x^2\psi + \partial_z^2\psi &=0\ ,\quad (x,z)\in \Omega(u(t))\ ,\quad t>0\ ,\label{psi1d}\\
\psi(t,x,z)&=\frac{1+z}{1+u(t,x)}\ ,\quad (x,z)\in  \partial\Omega(u(t))\,  ,\label{bcpsi1d}
\end{align}
in the two-dimensional time-varying domain
$$
\Omega(u(t)) := \left\{ (x,z)\in I\times (-1,\infty)\ :\ -1 < z < u(t,x) \right\}
$$
with $I:=(-1,1)$, coupled to the one-dimensional equation
\begin{equation}\label{u1d}
\begin{split}
\gamma^2\partial_t^2 u+\partial_t u +&\beta\partial_x^4 u-  \tau \partial_x^2 u \\
& =-\lambda\, \left(\ve^2|\partial_x\psi(t,x,u(t,x))|^2+ |\partial_z\psi(t,x,u(t,x))|^2\right)
\end{split}
\end{equation}
for $t>0$ and  $x\in I$ in which the self-stretching force has been neglected (i.e. $a=0$), supplemented with initial and boundary conditions
\begin{align}
u(t,\pm 1) &=\beta\partial_x u(t,\pm 1)=0\ , \quad t>0 \ , \label{bcu1d}\\
u(0,x) &=\gamma\partial_tu(0,x)=0\ ,\qquad x\in I\ .\label{icu1d}
\end{align}
As we shall see later, the mathematical analysis in this simplified geometry is already quite involved and we shall restrict in this survey to the particular case of equations~\eqref{psi1d}-\eqref{icu1d}. In Section~\ref{S6} we present extensions for the two-dimensional case~\eqref{psi}-\eqref{icu}. In principle, the same approach can be used in both cases. However, much weaker regularity is obtained in the two-dimensional setting and consequently, the analysis is less complete so far.

%%%%%%%%%%%%%%%%%%%%%%%
%%%%%%%%%%%%%%%%%%%%%%%
\subsection{Nonlinear Elasticity Model}\label{Sec2.3}
%%%%%%%%%%%%%%%%%%%%%%%
%%%%%%%%%%%%%%%%%%%%%%%

The previous linear elasticity model is restricted to small deformations neglecting curvature effects from the outset. The governing equation \eqref{u} (respectively \eqref{u1d}) for the plate deflection is thus semilinear as its highest order term $\beta\Delta^2 u$ is linear. As pointed out in \cite{BrP11} it may well be, however, that retaining gradient terms affects quantitatively and qualitatively the occurrence of the pull-in instability and is thus important in applications. To account for curvature effects in the simpler one-dimensional setting without self-stretching forces (i.e. $a=0$), the bending energy~$\mathcal{E}_b$ is replaced after rescaling by
\begin{equation*} 
\mathcal{E}_b(u)=\frac{Ybh^3\ve^2}{24}\int_{-1}^1  \left\vert\partial_{x}\left(\frac{\partial_{x}u({x})}{\sqrt{1+\ve^2(\partial_x {u}({x}))^2}}\right)\right\vert^2\, \sqrt{1+\ve^2(\partial_{x}u(x))^2}\,\rd x
\end{equation*}
involving the curvature of the graph of $u$ and its arc length element, while the stretching energy~$\mathcal{E}_s$ is proportional to the change of arc length and replaced by
\begin{equation*} 
\begin{split}
\mathcal{E}_s({u})&=TL^2\ve^2\int_{-1}^1  \left(\sqrt{1+\ve^2(\partial_{{x}}{u}(x))^2}-1\right)\, \rd x\,. %\\
%& \quad + \frac{pYL^4\ve^4}{8}\left(\int_{-1}^1  \left(\sqrt{1+\ve^2(\partial_{{x}}{u}(x))^2}-1\right)\, \rd x\right)^2 .
\end{split}
\end{equation*}
Analogously as before, the Euler-Lagrange equation of the total energy gives rise to the governing equation for the plate deflection. Consequently, taking curvature effects into account, equation \eqref{u1d} is replaced by
\begin{equation}\label{uquasilin}
\begin{split}
\gamma^2 \partial_t^2 u + \partial_t u+ \mathcal{K}(u)= - \lambda \left( \varepsilon^2 \vert \partial_{x} \psi(t,x,u(t,x)) \vert^2 + \vert\partial_z\psi(t,x,u(t,x)) \vert^2\right)
\end{split}
\end{equation}
with the quasilinear fourth-order operator 
\begin{equation*}\label{AA}
\begin{split}
\mathcal{K}(u):=&\beta\,  \partial_{x}^2\left(\frac{\partial_{x}^2u}{(1+\varepsilon^2(\partial_{x}u)^2)^{5/2}}\right) +  \frac{5}{2}\,\varepsilon^2\,\beta\, \partial_{x}\left(\frac{\partial_xu(\partial_{x}^2u)^2}{(1+\varepsilon^2(\partial_{x}u)^2)^{7/2}}\right)\\
&
-\tau\, \partial_{x}\left(\frac{\partial_{x}u}{(1+\varepsilon^2(\partial_{x}u)^2)^{1/2}}\right)\,.
\end{split}
\end{equation*}
The resulting problem \eqref{psi1d}, \eqref{bcpsi1d}, \eqref{uquasilin}, \eqref{bcu1d}, \eqref{icu1d} is studied in \cite{LW14c} and in the second-order case (i.e. $\beta=0$) in \cite{ELW15,ELW13}.

%%%%%%%%%%%%%%%%%%%%%%%
%%%%%%%%%%%%%%%%%%%%%%%
\subsection{Other Extensions}\label{Sec2.4}
%%%%%%%%%%%%%%%%%%%%%%%
%%%%%%%%%%%%%%%%%%%%%%%

There are various variants and extensions of the governing equations for a MEMS device in order to incorporate additional physical effects. Except for \cite{EsL16, Lie15, Liexx}, where the case of a non-uniform potential applied along the elastic plate is studied, and \cite{Koh13, Koh15a, Koh15b}, where a model involving two elastic plates suspended one above each other is considered, the analytical investigations regarding extended models have been concerned so far almost exclusively with vanishing aspect ratio models like \eqref{uSG}. For instance, spatial variations in the dielectric properties of the elastic plate which modify the strength of the Coulomb force may be included \cite{Pel01a} or the device may be embedded in a circuit in an attempt to control the pull-in instability \cite{PeT01}. If a simple capacitive control scheme is used, the right-hand side of equation~\eqref{uSG} is then replaced by a term
\begin{equation}
-\,  \frac{\lambda\, f(x,y)}{(1+u_0(x,y))^2 \left( \displaystyle{1+\chi\int_D\frac{\rd (x,y)}{1+u_0(x,y)}} \right)}\,,
\label{sccs}
\end{equation}
where $\chi$ is a ratio of capacitances in the system and $1/f(x,y)$ is proportional to the local dielectric permittivity. Observe that, when $\chi>0$, the additional nonlocal term does not depend on the spatial variables and may in some way be included in the parameter~$\lambda$. 

Actually, the common vanishing aspect ratio approximation leading to \eqref{uSG} is not uniformly valid in the spatial variables as fringing fields -- unaccounted for in the model (also with free boundary) -- may cause corrections in regions close to the boundary of the domain. Corner-corrected models are derived in \cite{Pel01b, PeB03, PeD05} and the right-hand side of equation~\eqref{uSG} is replaced by 
\begin{equation}
-  \frac{\lambda\, \left( 1+\delta\vert\nabla u_0(x,y)\vert^2 \right)}{(1+u_0(x,y))^2}\,, \quad \delta>0\ . \label{frfi}
\end{equation}

On a nanoscale, when the device is below a critical size, van der Waals and Casimir forces may become important which represent attraction of two uncharged bodies and may lead to a spontaneous collapse of the system with zero applied voltage. Such a phenomenon is considered in \cite{BPS08, GuZ04} for instance and leads to the addition of a term of the form 
\begin{equation}
- \frac{\mu}{(1+u_0)^{m}}\ , \qquad \mu>0\ , \label{vdWCa}
\end{equation} 
on the right-hand side of \eqref{uSG} with $m=3$ (van der Waals forces) or $m=4$ (Casimir forces).

Somewhat in the opposite direction, the addition of a term $\lambda\eta^{m-2} (1+u_0)^{-m}$, $m>2$, on the right-hand side of \eqref{uSG} is suggested in \cite{LLG14, LLG15, Lindsay16} with small parameter $\eta>0$ mimicking the effect of a thin insulating layer coating the ground plate and thereby preventing the touchdown phenomenon.

Finally, if a MEMS device is not fabricated to operate in a vacuumed capsule, the dynamics of the moving plate can be affected by the gas (or liquid) between the plates. This squeeze film phenomenon is taken into account by adding a pressure term on the right-hand side of \eqref{uSG} which itself is governed by the nonlinear Reynolds equation and depends on the deflection $u_0$.  We refer to \cite{FMCCS05, NaY04, BaY07, NYAR05, Jaz12} for a detailed description of the model which has been investigated only numerically so far.

%%%%%%%%%%%%%%%%%%%%%%%
%%%%%%%%%%%%%%%%%%%%%%%
\section{Tool Box}\label{S3}
%%%%%%%%%%%%%%%%%%%%%%%
%%%%%%%%%%%%%%%%%%%%%%%

A promising solving strategy for \eqref{psi1d}-\eqref{icu1d} is a decoupling of the equations for the electrostatic potential $\psi$ and the plate deflection $u$. Focusing first on the subproblem  \eqref{psi1d}-\eqref{bcpsi1d} for the electrostatic potential with an (almost) arbitrarily given, but fixed deflection $u$ and investigating the dependence of its solution $\psi=\psi_u$ on $u$,  is the basis to define
\begin{equation}\label{gg}
g_\ve(u)(x):= \ve^2|\partial_x\psi_u(x,u(x))|^2+ |\partial_z\psi_u(x,u(x))|^2\,,\quad x\in I=(-1,1)\,,
\end{equation}
which also reads
\begin{equation}\label{ggg}
g_\ve(u)(x) = \left( 1 + \ve^2 |\partial_x u(x)|^2 \right) |\partial_z\psi_u(x,u(x))|^2 \,,\quad x\in I\,,
\end{equation}
owing to the constant Dirichlet boundary condition \eqref{bcpsi1d}. The time dependence is omitted in \eqref{gg} and \eqref{ggg} since time acts only as a parameter in the elliptic subproblem \eqref{psi1d}-\eqref{bcpsi1d}. One can then reformulate \eqref{psi1d}-\eqref{icu1d} as a single nonlocal problem  
\begin{align}
\gamma^2\partial_t^2 u+\partial_t u +\beta\partial_x^4 u- &\tau \partial_x^2 u
 =-\lambda\, g_\ve(u)\,,\qquad t>0\,,\quad x\in I\,, \notag\\
u(t,\pm1)&=\beta\partial_x u(t,\pm1)=0\ , \quad t>0\ ,\, \label{u1dd}\\
u(0,x)&= \gamma\partial_tu(0,x)=0\ ,\qquad x\in I\ ,\notag
\end{align}
only involving the plate deflection $u$. In the following we shall collect some auxiliary results needed for this approach. In Section~\ref{Sec3.1} we first provide the reduction to~\eqref{u1dd} and state important properties of the function $g_\ve$. In Section~\ref{Sec3.2} we then set the stage for a variational formulation of~\eqref{u1dd}. Existence of solutions to~\eqref{u1dd} and their qualitative aspects shall be studied later in Section~\ref{S4} and Section~\ref{S5}.

%%%%%%%%%%%%%%%%%%%%%%%
%%%%%%%%%%%%%%%%%%%%%%%
\subsection{The Electrostatic Potential}\label{Sec3.1}
%%%%%%%%%%%%%%%%%%%%%%%
%%%%%%%%%%%%%%%%%%%%%%%

For an arbitrarily fixed deflection $u$, the subproblem \eqref{psi1d}-\eqref{bcpsi1d} is an elliptic equation with Dirichlet boundary data. Thus  only minimal regularity assumptions on $u$ -- including its continuity along with the constraint $u(\pm 1)=0$ guaranteeing an open domain $\Omega(u)$ -- are needed in order to ensure the existence of a unique weak solution $\psi_u$ to \eqref{psi1d}-\eqref{bcpsi1d} in the Sobolev space $H^1(\Omega(u))$ by the well-known Lax-Milgram Theorem.
However, the provided regularity for $\psi_u$ is far from being sufficient to give a meaning to the square of its gradient trace $\vert\nabla \psi_u(\cdot,u)\vert^2$ needed later to define $g_\ve(u)$ in \eqref{gg}. Moreover, no information is so obtained on the dependence of this term on the given deflection $u$. For this reason, a beneficial approach is to restate the Laplace equation \eqref{psi1d}-\eqref{bcpsi1d} for $\psi_u$ as a boundary value problem with zero Dirichlet data of the form
\begin{equation}\label{PLPhi}
-\mathcal{L}_u \Phi_u=f_u \ \text{ in }\ \Omega\,,\quad \Phi_u=0\ \text{ on }\ \partial\Omega
\end{equation}
for the transformed electrostatic potential
\begin{equation}\label{PLphi}
\Phi_u(x,\eta) := \psi_u\big( x, (1+u(x))\eta - 1 \big) -\eta \ , \quad (x,\eta)\in\Omega\, ,
\end{equation}
in the fixed rectangle $\Omega:=I\times (0,1)$.
The differential operator $-\mathcal{L}_u$ has coefficients depending on $u$ and its $x$-deri\-va\-tives up to second-order and being singular at touchdown points $x\in I$ where \mbox{$u(x)=-1$}. However, as long as $u(x)>-1$  for all $x\in I$, the operator $-\mathcal{L}_u$ is elliptic. We thus restrict to deflections $u$ belonging to
$$
S_q^\alpha(\kappa):=\left\{v\in W_{q,D}^\alpha(I)\,;\,  \| v\|_{W_q^\alpha(I)}<\kappa^{-1}\,,\, v>-1+\kappa\right\}
$$
for some $\alpha>1/q$ and $\kappa\in (0,1)$, where $q\in [1,\infty]$.  Here 
\begin{align*}
W_{q,D}^\alpha(I) := \left\{
\begin{array}{ccl}
W_{q}^\alpha(I) & \text{ if } & \alpha\le 1/q\ , \\
& & \\
\{ v\in W_{q}^\alpha(I)\, ;\, v(\pm 1)=0 \} & \text{ if } & 1/q< \alpha\le (q+1)/q\ , \\
& & \\
\{ v\in W_{q}^\alpha(I)\, ;\, v(\pm 1)=\beta \partial_x v(\pm 1) = 0 \} & \text{ if } & (q+1)/q< \alpha\ .
\end{array}
\right.
\end{align*} 
We stress here that, when $\beta=0$ and $\alpha>1/q$, the first-order boundary conditions are irrelevant and $W_{q,D}^\alpha(I)$ is simply the subspace of the classical Sobolev space $W_{q}^\alpha(I)$ consisting of those functions $v$ satisfying only $v(\pm 1)=0$. Observe that if $\alpha>1/q$, then $W_q^\alpha(I)$ is continuously embedded in $C(\bar{I})$ so that $\Omega(u)$ is a well-defined open set of $\mathbb{R}^2$. For simplicity we put  
$$
S_q^\alpha:= \mathop{\bigcup}_{\kappa\in (0,1)} S_q^\alpha(\kappa)\ , \quad q\in [1,\infty]\ , \quad \alpha>\frac{1}{q}\ .
$$
Depending on the smoothness of the deflection $u$, we derive different regularity properties of the electrostatic potential~$\psi_u$. Introducing the Sobolev spaces
$$
H_D^s(\Omega) := \left\{ \phi\in H^s(\Omega)\, ;\, \phi=0 \;\text{ on }\; \partial\Omega \right\}\ , \quad s>\frac{1}{2}\ ,
$$
we can use elliptic regularity theory to show that, for sufficiently smooth deflections $u$, the operator $-\mathcal{L}_u$ is bounded and invertible from $H_D^2(\Omega)$ onto $L_2(\Omega)$ so that
 the transformed potential $\Phi_u$ solving \eqref{PLPhi} belongs to $H_D^2(\Omega)$ and depends smoothly on $u$. For less regular deflections $u$, the operator  $-\mathcal{L}_u$ no longer maps $H_D^2(\Omega)$ into $L_2(\Omega)$. Still it is a bounded and invertible operator from $H_D^1(\Omega)$ onto $(H_D^{1}(\Omega))'$. A careful analysis reveals that its inverse maps $L_2(\Omega)$ into the anisotropic space 
$$
X(\Omega):=\{\theta\in H_D^1(\Omega)\,;\, \partial_\eta\theta\in H^1(\Omega)\}\,,
$$
and $\Phi_u\in X(\Omega)$ is weakly sequentially continuous with respect to $u$. 
The so-derived key properties in terms of the electrostatic potential $\psi_u$ are collected in the following result.

%%%%%%%%%%%%%%%%%%%%%%%
\begin{prop} [Nonlocal Nonlinearity \hspace{-.1mm}\cite{ELW14,LW_Multi}] \label{P1}
\begin{enumerate}
\item[(i)] For each $u\in S_2^1$ there is a unique weak solution $\psi_u \in H_D^{1}(\Omega(u))$ to \eqref{psi1d}-\eqref{bcpsi1d}.
\item[(ii)] Assume that $\beta>0$ and consider $\alpha\in (3/2,2]$. If $u\in S_2^{\alpha}$, then $\psi_u \in H^{\alpha}(\Omega(u))$, and the mapping
$$
g_\ve: S_2^{\alpha}\rightarrow H^\sigma(I)\ ,\quad u\mapsto \varepsilon^2 |\partial_x \psi_u(\cdot,u)|^2 + |\partial_z \psi_u(\cdot,u)|^2
$$
is continuous for each $\sigma\in [0,1/2)$.
\item[(iii)] Let $\kappa\in (0,1)$, $q\in (2,\infty]$, and $\sigma\in [0,1/2)$. Then 
$g_\ve:S_q^{2}(\kappa)\rightarrow H^\sigma(I)$ is globally Lipschitz continuous, bounded, and analytic.
\end{enumerate}
\end{prop}
%%%%%%%%%%%%%%%%%%%%%%%

The proofs of the statements~(ii) and~(iii) in Proposition~\ref{P1} are of a completely different nature. While (iii) is a consequence of the regularity theory for elliptic equations with sufficiently smooth coefficients \cite{LU68}, the proof of (ii) is more involved and requires to exploit carefully the specific structure of the right-hand side $f_u$ of \eqref{PLPhi} and the vanishing of $\partial_x u(\pm 1)$ to derive estimates in $X(\Omega)$ for $\Phi_u$. In particular, it is restricted to the fourth-order case $\beta>0$. Note that $g_\ve$ acts in statement (ii) as a mapping of (roughly) order one while in (iii) it is of order $3/2$, but with significant additional properties.

The particular geometry of $\Omega(u)$ with its corners apparently precludes a higher regularity than $H^2(\Omega(u))$ for the electrostatic potential $\psi_u$. This is a far-reaching technical obstruction for the analysis. 
Also, as pointed out in the introduction, convexity and monotonicity properties of the electrostatic potential  or of the function~$g_\ve$ --~obvious and important for the study of the vanishing ratio equation~\eqref{uSG}~-- are not known in the present situation. 

%%%%%%%%%%%%%%%%%%%%%%%
%%%%%%%%%%%%%%%%%%%%%%%
\subsection{The Electrostatic Energy}\label{Sec3.2}
%%%%%%%%%%%%%%%%%%%%%%%
%%%%%%%%%%%%%%%%%%%%%%%

A useful tool in the investigation of qualitative aspects for solutions to equation~\eqref{u1dd} is its gradient flow structure. In Section~\ref{S2} we have seen that (the stationary) equation~\eqref{u1dd} is {\it formally} the Euler-Lagrange equation of the (rescaled) total energy
\begin{equation}\label{EE}
\mathcal{E}(u):=\mathcal{E}_m(u)-\lambda \mathcal{E}_e(u)
\end{equation}
involving the mechanical energy 
\begin{equation*}
\mathcal{E}_m(u) := \frac{\beta}{2} \|\partial_x^2 u \|_{2}^2 + \frac{\tau}{2} \|\partial_x u \|_{2}^2 
\end{equation*}
and the electrostatic energy 
\begin{equation*}
\mathcal{E}_e(u) := \int_{\Omega(u)} \left( \varepsilon^2 |\partial_x  \psi_u|^2 + |\partial_z  \psi_u|^2 \right)\ \mathrm{d}(x,z)\, ,
\end{equation*}
where $\psi_u$ is the electrostatic potential associated with a given deflection $u$ and provided by Proposition~\ref{P1}. 
The stumbling block in this direction is obviously the computation of the Fr\'echet derivative with respect to $u$ of the elastic energy $\mathcal{E}_e$, which involves the implicit dependence of $\psi_u$  on $u$ via the domain $\Omega(u)$. Nevertheless, this difficulty can be overcome by interpreting the derivative as the shape derivative of the Dirichlet integral of $\psi_u$, which can be computed by shape optimization arguments~\cite{HeP05} and a transformation to the fixed rectangle~$\Omega$ introduced in the previous section. More precisely, the rigorous result is the following.

%%%%%%%%%%%%%%%%%%%%%%%
\begin{prop} [Derivative of the Electrostatic Energy \hspace{-.1mm}\cite{LW_Multi}] \label{P2}
 The functional $\mathcal{E}_e$ belongs to $C(S_2^1)$, and if  $\beta>0$ and $\alpha\in (3/2,2]$, then $\mathcal{E}_e\in C^1(S_2^{\alpha})$  with Fr\'echet derivative $\partial_u \mathcal{E}_e(u) = -g_\ve(u)$ for $u\in S_2^{\alpha}$.
\end{prop}
%%%%%%%%%%%%%%%%%%%%%%%

According to Proposition~\ref{P2}, the Fr\'echet derivative of $\mathcal{E}_e$ in $S_2^2$ is non-negative when $\beta>0$, which indicates a monotonicity property of $\mathcal{E}_e$. The latter is actually valid in a more general setting and is somewhat reminiscent of that of the classical Dirichlet energy \cite{HeP05}.
 
%%%%%%%%%%%%%%%%%%%%%%%
\begin{prop}[Monotonicity of the Electrostatic Energy \hspace{-.1mm}\cite{LW_Multi}]\label{P3}
Consider two functions $u_1$ and $u_2$ in $S_2^1$ such that $u_1\le u_2$. Then $\mathcal{E}_e(u_2)\le \mathcal{E}_e(u_1)$. 
\end{prop}
%%%%%%%%%%%%%%%%%%%%%%%

Unfortunately, a monotonicity (or concavity) property of the derivative of $\mathcal{E}_e$, that is, of the right-hand side $-g_\ve$ of the evolution equation \eqref{u1dd}, is not known. This is in clear contrast to the vanishing aspect ratio equation \eqref{uSG}, where such a monotonicity is obvious and of utmost importance in the analysis. We shall see later in Section~\ref{S4} and Section~\ref{S5}, however, how the information provided by Proposition~\ref{P2} and Proposition~\ref{P3} on the variational structure of \eqref{u1dd} can be used to gain insight into qualitative aspects. Still, some care is needed if one wishes to exploit this variational structure. Indeed, it is quite clear that $\mathcal{E}_e$ is unbounded from below on $S_2^1$ and  quantified by the estimates
\begin{equation}\label{MarcusMiller}
\int_{-1}^1 \frac{\mathrm{d}x}{1+u(x)} \le \mathcal{E}_e(u) \le \int_{-1}^1 \frac{ 1 + \varepsilon^2 |\partial_x u(x)|^2}{1+u(x)}\,\mathrm{d}x\,,\quad u\in S_2^1\,  .
\end{equation}
To prove the first inequality in \eqref{MarcusMiller} it suffices to note that \eqref{bcpsi1d} and the Cauchy-Schwarz inequality imply
\begin{align*}
\int_{-1}^1\frac{1}{1+u(x)}\, \rd x& = \int_{-1}^1\frac{\left( \psi_u(x,u(x)) - \psi_u(x,-1) \right)^2}{1+u(x)}\,\rd x\\
& = \int_{-1}^1\frac{1}{1+u(x)} \left( \int_{-1}^{u(x)} \partial_z \psi_u(x,z)\ \mathrm{d}z \right)^2 \rd x \\
& \le \int_{-1}^1\int_{-1}^{u(x)} \left( \partial_z \psi_u(x,z) \right)^2\ \mathrm{d}z\, \rd x\ . 
\end{align*}
The second one is a consequence of the variational inequality satisfied by the weak solution $\psi_u$ to \eqref{psi1d}-\eqref{bcpsi1d} (see \cite[Lemma~2.8]{LW_Multi} for details). Clearly, $\mathcal{E}_e(u)$ is not finite for any deflection $u\in C^1(\bar{I})$ satisfying $u\ge -1$ and $u(x_0)=-1$ for some $x_0\in I$.  It is worth pointing out that the vanishing aspect ratio $\ve=0$ leads to equality on both sides in \eqref{MarcusMiller} and thus, the bounds on $\mathcal{E}_e$ when $\ve >0$ seem to be nearly optimal. 
That $\mathcal{E}_e$ features a singularity for certain touchdown deflections becomes also apparent from the striking identity \cite{LW_Multi}
\begin{equation}
\mathcal{E}_e(u) =  2 - \int_{-1}^1 u(x) \left( 1 + \varepsilon^2 |\partial_x u(x)|^2 \right)\, \partial_z \psi_u(x,u(x))\, \mathrm{d}x\,, \label{b19}
\end{equation} 
which is derived from \eqref{psi}-\eqref{bcpsi} by means of Green's formula and indicates that the blowup of $\mathcal{E}_e(u)$ resulting from the touchdown phenomenon is accompanied by a singular behavior of $x\mapsto \partial_z \psi_u(x,u(x))$ as already mentioned.

%%%%%%%%%%%%%%%%%%%%%%%
%%%%%%%%%%%%%%%%%%%%%%%
\subsection{The Total Energy}\label{Sec3.25}
%%%%%%%%%%%%%%%%%%%%%%%
%%%%%%%%%%%%%%%%%%%%%%%

Proposition~\ref{P2} implies that the stationary solutions to \eqref{u1dd} are the critical points of the energy $\mathcal{E} = \mathcal{E}_m - \lambda \mathcal{E}_e$ defined in \eqref{EE} and that \eqref{u1dd} may be seen as a gradient flow associated with $\mathcal{E}$. In particular, $\mathcal{E}$ is decreasing along sufficiently smooth solutions to the dynamic problem \eqref{u1dd}.

%%%%%%%%%%%%%%%%%%%%%%%
\begin{cor} [Total Energy \hspace{-.1mm}\cite{LW14a}] \label{C1}
\begin{itemize}
\item[(i)] Let $u\in S_2^4$. Then $u$ is a stationary solution to \eqref{u1dd} if and only if it is a critical point of the total energy $\mathcal{E}$.
\item[(ii)] Let $u\in C([0,T], S_2^{2+\nu})\cap C^1([0,T],H_D^\nu (I))$ solve \eqref{u1dd} for some $\nu>0$ and $T>0$. Then
\begin{equation}
\mathcal{E}(u(t)) + \frac{\gamma^2}{2}\| \partial_t u(t)\|_{2}^2+\int_0^t \|\partial_t u(s)\|_{2}^2\ \mathrm{d}s = \mathcal{E}(0)\label{rc2b}
\end{equation}
for $t\in [0,T]$.
\end{itemize}
\end{cor}
%%%%%%%%%%%%%%%%%%%%%%%

Retrieving tractable information from this variational structure is, however, not obvious since the total energy $\mathcal{E}$ is clearly the sum of terms with different signs and the possible pull-in instability becomes manifest in the {\it non-coercivity} of the energy $\mathcal{E}$. Nevertheless, we shall use it later on to construct stationary solutions by a constrained variational approach, thereby bypassing the non-coercivity of the energy. Also, combining \eqref{MarcusMiller} and \eqref{rc2b} provides regularity estimates on $u$ as soon as one knows that it is separated from $-1$. 

%%%%%%%%%%%%%%%%%%%%%%%
%%%%%%%%%%%%%%%%%%%%%%%
\subsection{Maximum Principles}\label{Sec3.3}
%%%%%%%%%%%%%%%%%%%%%%%
%%%%%%%%%%%%%%%%%%%%%%%

When available, the celebrated maximum principles are central tools in the analysis of differential equations and guarantee, roughly speaking, non-negativity (or even positivity) of solutions provided the boundary and/or initial data are non-negative and the equations satisfy suitable properties. Thus, it is not surprising that the deepest results regarding MEMS equations are for situations where a maximum principle holds true. While such a tool is well-established for second-order (elliptic) differential equations, less is known for fourth-order problems to which \eqref{u1dd} belongs when~$\beta>0$. We are mainly interested in the fourth-order operator $\beta \Delta^2 - \tau \Delta$ in a bounded smooth domain $D$ of $\mathbb{R}^d$, $d\ge 1$, with $\beta>0$ and $\tau\ge 0$. Let us point out, though, that the validity of the maximum principle heavily depends on the boundary conditions and the domain itself. The solution to the fourth-order equation with {\it pinned} boundary conditions \eqref{bcubcuPinned}, i.e.,
\begin{equation}
\beta \Delta^2 \phi - \tau \Delta\phi = f \;\text{ in }\; D\ , \qquad \phi=\Delta \phi = 0 \;\text{ on }\; \partial D\ , \label{i1}
\end{equation}
satisfies $\phi(x)>0$, $x\in D$, provided that $f$ is a smooth and nonnegative function in $\bar{D}$ with $f\not\equiv 0$. This is readily seen by writing \eqref{i1} as a system of two second-order equations, each of them involving a second-order linear elliptic operator with homogeneous Dirichlet boundary conditions on $\partial D$, so that the maximum principle applies.\footnote{Note that this is not true in non-smooth domains, see \cite{GGS10}.} A similar result, however, fails to be true in general when the pinned boundary conditions \eqref{bcubcuPinned} are replaced by the {\it clamped} boundary conditions \eqref{bcu} in which case the same transformation no longer works as one equation of the second-order system comes without and the other with overdetermined boundary conditions. 
Nevertheless, it was observed by Boggio~\cite{Bog05} that, when $D$ is the unit ball $\mathbb{B}_1$ in $\mathbb{R}^d$, $d\ge 1$, and $\tau=0$, the associated Green function can be computed explicitly and is positive in $\mathbb{B}_1$.  An obvious consequence of this positivity property is that solutions to 
\begin{equation}
\Delta^2 \phi = f \;\text{ in }\; \mathbb{B}_1\ , \qquad \phi=\partial_\nu \phi = 0 \;\text{ on }\; \partial \mathbb{B}_1\ , \label{i2}
\end{equation}
 with a non-negative right-hand side $f\not\equiv 0$ are positive in $\mathbb{B}_1$. Actually, Boggio's celebrated result provides the impetus for several further studies, including the extension of the maximum principle to other domains \cite{EnP96, GrR10, GrS96, Had08, Sas07} and other operators \cite{Gru02, Boggio, Owe97, Sch80}. See also the monograph~\cite{GGS10} for a detailed discussion of positivity properties of higher-order elliptic operators. 

Our particular interest for the maximum principle with regard to MEMS equations is the inclusion of second-order terms in the differential operator accounting for stretching. As the following result shows, this is possible in one dimension or in radial symmetry.

%%%%%%%%%%%%%%%%%%%%%%%
\begin{prop}[Maximum Principle \hspace{-.1mm}\cite{Bog05, Gru02, Owe97, Boggio}]\label{P4}
Let $\beta>0$ and $\tau\ge 0$. 
\begin{enumerate}
\item[(i)] If $f: \bar{\mathbb{B}}_1\rightarrow [0,\infty)$ is smooth with $f\not\equiv 0$, then the solution to
\begin{equation*}
\beta\Delta^2\phi = f\ \text{ in } \mathbb{B}_1\ ,\quad \phi= \partial_\nu\phi=0\ \text{ on } \partial \mathbb{B}_1\ ,
\end{equation*}
satisfies $\phi(x)>0$, $x\in \mathbb{B}_1$. 
\item[(ii)] If $f: \bar{\mathbb{B}}_1\rightarrow [0,\infty)$ is smooth with $f\not\equiv 0$ and, in addition, radially symmetric when $d>1$, then the solution to
\begin{equation}\label{gutemine}
\beta\Delta^2\phi -\tau\Delta\phi =f\ \text{ in } \mathbb{B}_1\ ,\quad \phi= \partial_\nu\phi=0\ \text{ on } \partial \mathbb{B}_1\ ,
\end{equation}
satisfies $\phi(x)>0$, $x\in \mathbb{B}_1$. 
\item[(iii)] The eigenvalue problem 
\begin{equation}\label{obelix}
\beta\Delta^2\zeta -\tau\Delta\zeta =\mu \zeta\ \text{ in } \mathbb{B}_1\ ,\quad \zeta= \partial_\nu\zeta=0\ \text{ on } \partial \mathbb{B}_1\ ,
\end{equation}
has an eigenvalue $\mu_1>0$ with a corresponding radially symmetric (normalized) eigenfunction $\zeta_1>0$ in $\mathbb{B}_1$, and $\mu_1$ is the only eigenvalue with a positive radially symmetric eigenfunction. Moreover, any radially symmetric eigenfunction corresponding to the eigenvalue $\mu_1$ is a scalar multiple of $\zeta_1$.
\end{enumerate}
\end{prop}
%%%%%%%%%%%%%%%%%%%%%%%

As already mentioned the first statement of Proposition~\ref{P4} is proved in \cite{Bog05}. The maximum principle stated in Proposition~\ref{P4}~(ii) is established in \cite{Owe97} for $d=1$  by means of an explicit computation of the associated Green function and later extended in \cite{Gru02} (by a different proof relying on the maximum principle for second-order equations) also to include third order terms. The higher-dimensional case in radial symmetry and the application to the principal eigenvalue mentioned in Proposition~\ref{P4}~(iii) were recently shown for a wider class of of fourth-order operators that can be written as a composition of two second-order operators~\cite{Boggio}.

It is important to point out that the maximum principle fails for
\begin{equation*}
\beta\Delta^2\phi -\tau\Delta\phi +p \phi=f\ \text{ in } \mathbb{B}_1\ ,\quad \phi= \partial_\nu\phi=0\ \text{ on } \partial \mathbb{B}_1\ ,
\end{equation*}
when $p>0$ is too large. The maximum principle for evolution problems of the form \eqref{u1dd} is thus not valid, not even if $\gamma=0$. Therefore, when $\beta>0$, this central tool can mainly be helpful for the study of stationary solutions to \eqref{u1dd}.

%%%%%%%%%%%%%%%%%%%%%%%
%%%%%%%%%%%%%%%%%%%%%%%
\section{Stationary Problem}\label{S4}
%%%%%%%%%%%%%%%%%%%%%%%
%%%%%%%%%%%%%%%%%%%%%%%

We now focus our attention on classical stationary solutions to MEMS equations which in particular have a minimum  value strictly greater than $-1$. As pointed out in the introduction the pull-in instability of MEMS devices is supposed to be revealed by a sharp threshold \mbox{$\lambda_\ve^{stat}>0$} for the voltage parameter $\lambda$ above which no such stationary solution exists while below there is at least one. In Section~\ref{Sec4.1} we collect  the present knowledge on this matter for the free boundary problem \eqref{psi1d}-\eqref{icu1d}. As we shall see, however, the picture regarding the pull-in voltage is not yet completely established. For the vanishing aspect ratio model \eqref{uSG}, the study  is understandably more developed as briefly described in Section~\ref{Sec4.2}. 

%%%%%%%%%%%%%%%%%%%%%%%
%%%%%%%%%%%%%%%%%%%%%%%
\subsection{Stationary Solutions to the Free Boundary Problem}\label{Sec4.1}
%%%%%%%%%%%%%%%%%%%%%%%
%%%%%%%%%%%%%%%%%%%%%%%

If not stated otherwise, $\ve>0$ is fixed in this section and we often omit the dependence on this parameter. 

Stationary solutions $(u,\psi)$ to the free boundary problem~\eqref{psi1d}-\eqref{icu1d} satisfy equivalently
\begin{equation}\label{u1ddstat}
\begin{split}
\beta\partial_x^4 u-  \tau \partial_x^2 u 
 &=-\lambda\, g_\ve(u)\,,\qquad x\in I\,, \\
u(\pm 1)&=\beta\partial_x u(\pm 1)=0\ ,
\end{split}
\end{equation}
where $g_\ve(u)$ with $\psi=\psi_u$ is given in \eqref{gg}, its properties being stated in Proposition~\ref{P1}. Of course, $u>-1$ in $[-1,1]$.

%%%%%%%%%%%%%%%%%%%%%%%
%%%%%%%%%%%%%%%%%%%%%%%
\subsubsection{Non-Existence}
%%%%%%%%%%%%%%%%%%%%%%%
%%%%%%%%%%%%%%%%%%%%%%%

As expected on physical grounds, stationary solutions do not exist if the applied voltage is too large. This is rather easily seen in the second-order case, that is, when $\beta=0$ in~\eqref{u1ddstat}. Indeed, in this case, any stationary solution $u$ is clearly {\em convex}.  Consequently, the function $(x,z)\mapsto 1+z-u(x)$ is a supersolution to \eqref{psi1d}-\eqref{bcpsi1d}, and the comparison principle implies
$$
\psi_u(x,z)\le 1+z-u(x)\ , \quad (x,z)\in \overline{\Omega(u)}\ ,
$$
hence
\begin{equation}\label{convex}
\partial_z \psi_u(x,u(x)) \ge 1\ ,\quad x\in I\ ,
\end{equation}
according to \eqref{bcpsi1d}. Consequently, $g_\ve(u)\ge 1$ in $I$ and we infer from \eqref{u1ddstat} (with $\beta=0$) that $\tau \partial_x^2 u \ge \lambda$ in $I$. Multiplying this inequality by the nonnegative eigenfunction $\phi_1(x)=\cos(\pi x/2)$ of the Laplacian $-\partial_x^2$ in $I$ with homogeneous Dirichlet boundary conditions and integrating over $I$ readily give
$$
\lambda \int_{-1}^1 \phi_1\ \rd x \le \tau \int_{-1}^1 u \partial_x^2\phi_1\ \rd x = - \frac{\tau\pi^2}{4} \int_{-1}^1 \phi_1 u \ \rd x \le \frac{\tau\pi^2}{4} \int_{-1}^1 \phi_1 \ \rd x \ ,
$$
since $u>-1$. This shows that $\lambda$ cannot exceed $\tau\pi^2/4$ (this bound can be improved, though).

Unfortunately, in the fourth-order case $\beta>0$, stationary deflections with clamped boundaries are not convex and \eqref{convex} thus no longer holds. Nevertheless, the non-existence result is still true, but requires a more sophisticated proof. It is based on a nonlinear variant of the eigenfunction method involving the positive eigenfunction $\zeta_1$ from Proposition~\ref{P4} associated with the operator $\beta\partial_x^4 - \tau\partial_x^2$ subject to clamped boundary conditions  and normalized by $\|\zeta_1\|_{1}=1$. 

%%%%%%%%%%%%%%%%%%%%%%%
\begin{theorem}[Non-Existence \cite{LW13, LW14a}]\label{T2}
Let $(\beta,\tau)\in [0,\infty)^2\setminus\{(0,0)\}$.
There exist $\ve_*>0$ (with $\ve_*=\infty$ if $\beta=0$) and \mbox{${\lambda_{ns}}: (0,\ve_*)\to (0,\infty)$} such that there is no stationary solution $(u,\psi_u)$ to \eqref{psi1d}-\eqref{icu1d} for $\ve\in (0,\ve_*)$ and $\lambda>{\lambda_{ns}}(\ve)$.
\end{theorem}
%%%%%%%%%%%%%%%%%%%%%%%

Of course, in Theorem~\ref{T2} it is expected that $\ve_*=\infty$ also if $\beta>0$.

\begin{proof} Let $u$ be a stationary solution to \eqref{u1ddstat} with $\beta>0$ and recall that it satisfies $-1<u<0$ in $I$ by Proposition~\ref{P4}. 
Hence the solution $\mathcal{U}\in H^2_D(I)$ to 
$$
- \partial_x^2 \mathcal{U} = u \;\mbox{ in }\; I\ , \quad \mathcal{U}(\pm 1)=0\ ,
$$
satisfies $-1/2 \le \mathcal{U}\le 0$. We multiply \eqref{u1ddstat} by  $\zeta_1(1+\alpha \mathcal{U})$ with $\alpha = \min{\{1,\ve^2\}}\in (0,1]$ and integrate over $I$. After integrating several times by parts and using the properties of $u$, $\mathcal{U}$, and $\zeta_1$ including the lower bound $1+  \alpha \mathcal{U}\ge 1/2$, we obtain
$$
\frac{\lambda}{2} \int_{-1}^1 \zeta_1 g_\ve(u) \ \rd x \le \lambda \int_{-1}^1 (1+ \alpha\mathcal{U}) \zeta_1 g_\ve(u)\ \rd x \le K_1 - \alpha \beta \int_{-1}^1 \zeta_1 |\partial_x u|^2\ \rd x 
$$
for some $K_1>0$ depending only on $\beta$ and $\tau$. Next, setting
$$
\gamma_m(x) := \partial_z \psi_u(x,u(x))\ ,\quad x\in I\ ,
$$ 
we infer from Young's inequality that, for $\delta\in (0,1)$, 
$$
\frac{\lambda}{2} \int_{-1}^1 \zeta_1 g_\ve(u) \ \rd x \ge \frac{\lambda}{\delta} \int_{-1}^1 \zeta_1 (1+\varepsilon^2 |\partial_x u|^2) \gamma_m\ \rd x - \frac{\lambda}{2\delta^2} \left( 1 + \ve^2 \int_{-1}^1 \zeta_1 |\partial_x u|^2 \ \rd x \right)\ ,
$$
while computations in the spirit of \eqref{MarcusMiller} and \eqref{b19} and the non-positivity of $u$ give
$$
\int_{\Omega(u)} \zeta_1 |\partial_z \psi_u|^2 \ \rd (x,z) \ge \int_{-1}^1 \frac{\zeta_1}{1+u}\ \rd x \ge 1
$$
and
$$
\int_{-1}^1 \zeta_1 (1+\varepsilon^2 |\partial_x u|^2) \gamma_m\ \rd x \ge \int_{\Omega(u)} \zeta_1 \left( \ve^2 |\partial_x \psi_u|^2 + |\partial_z \psi_u|^2 \right)\, \rd (x,z) - \ve^2 \|\partial_x^2\zeta_1\|_{1} \ .
$$
Combining these inequalities leads to
$$
K_1 \ge \frac{\lambda}{\delta} \left[ 1 - \frac{1}{2\delta} - \ve^2 \|\partial_x^2\zeta_1\|_{1} \right] + \left( \alpha \beta - \frac{\lambda \ve^2}{2\delta^2} \right) \int_{-1}^1 \zeta_1 |\partial_x u|^2 \ \rd x
$$
and the choice $\delta^2 := \lambda\ve^2/(2\alpha\beta)$ cancels the last term of the right-hand side of the above inequality. Consequently,
$$
K_1 + \frac{ \alpha \beta}{\ve^2} \ge \frac{\sqrt{2 \alpha\beta}}{\ve} \left[ 1 - \ve^2 \|\partial_x^2\zeta_1\|_{1} \right] \sqrt{\lambda}\ ,
$$
from which we readily deduce that $\lambda$ cannot exceed a threshold value depending on $\ve$ provided $\ve<\ve_* := 1/\sqrt{\|\partial_x^2\zeta_1\|_{1}}$.  In particular, for $\ve \in (0 , \min\{1,\ve_*/2\})$ we conclude from the above bound and the choice of $\alpha$ that 
$$
\lambda \le \frac{8 (K_1+\beta)^2}{9\beta}\ ,
$$
which does not depend on $\ve$.
\end{proof}

%%%%%%%%%%%%%%%%%%%%%%%
%%%%%%%%%%%%%%%%%%%%%%%
\subsubsection{Existence}
%%%%%%%%%%%%%%%%%%%%%%%
%%%%%%%%%%%%%%%%%%%%%%%

To establish the existence of stationary solutions for small values of $\lambda$ we introduce the operator $A:=\beta\partial_x^4 - \tau\partial_x^2  \in \mathcal{L}(D(A),L_p(I))$ with $D(A):=H_D^4(I)$ and $p=2$ if $\beta>0$ and $D(A):=W_{p,D}^2(I)$ with $p>2$ if $\beta=0$ and $\tau>0$. Note that $A$ is invertible. Given $q>2$, Proposition~\ref{P1} implies that solutions to \eqref{u1ddstat} are the zeros of the smooth function
$$
F:\mathbb{R}\times \big(D(A)\cap S_q^2(\kappa)\big)\rightarrow D(A)\ ,\quad (\lambda,v)\mapsto v + \lambda A^{-1} g_\ve(v)\ .
$$
The Implicit Function Theorem yields a smooth curve $(\lambda, U_\lambda)$ in $\R\times D(A)$  of zeros starting at $(\lambda,v)=(0,0)$ which can be maximally extended as long as the principal eigenvalue of the associated linearized operator $A+\lambda \partial_u g_\ve(U_\lambda)$ in $L_p(I)$ is positive. 

%%%%%%%%%%%%%%%%%%%%%%%
\begin{theorem}[Existence \cite{LW14a, ELW14}]\label{T1}
Let $(\beta,\tau)\in\R_+^2\setminus\{(0,0)\}$. There is $\lambda_s(\varepsilon)>0$ such that for each $\lambda\in (0,\lambda_s(\varepsilon))$ there exists a stationary solution $(U_\lambda,\Psi_{U_\lambda})$ to \eqref{psi1d}-\eqref{icu1d}  satisfying $-1<U_\lambda< 0$ in $I$ and the associated  linearized operator $A+\lambda \partial_u g_\ve(U_\lambda)$ in $L_p(I)$ has a positive principal eigenvalue. In addition, $U_\lambda =U_\lambda(x)$ and $\Psi_{U_\lambda} =\Psi_{U_\lambda}(x,z)$ are even with respect to $x\in I$. 
\end{theorem}
%%%%%%%%%%%%%%%%%%%%%%%

The sign property of $U_\lambda$ is a consequence of Proposition~\ref{P4}.  Actually, one can show that these stationary solutions are locally asymptotically stable for the evolution \cite{LW14a, ELW14}. Obviously, the two critical values found in Theorem~\ref{T2} and Theorem~\ref{T1} satisfy $\lambda_{ns}(\ve)\ge\lambda_\ve^{stat}\ge\lambda_s(\ve)$, but otherwise no further information is yet available. In particular, the occurrence of a sharp threshold in form of a pull-in instability (i.e. $\lambda_{ns}(\ve)=\lambda_s(\ve)$)  is an open problem. Even the simpler problem of identifying the behavior of $(U_\lambda, \Psi_{U_\lambda})$ as $\lambda\to\lambda_s(\ve)$ seems currently out of reach.

%%%%%%%%%%%%%%%%%%%%%%%
%%%%%%%%%%%%%%%%%%%%%%%
\subsubsection{Multiplicity}
%%%%%%%%%%%%%%%%%%%%%%%
%%%%%%%%%%%%%%%%%%%%%%%

Theorem~\ref{T1} leaves open the question whether more than one stationary solution exists for small values of $\lambda$ (as suggested by what is known for the vanishing aspect ratio equation, see Section~\ref{Sec4.2} below). 
Based on the variational structure of \eqref{u1ddstat} pointed out in Proposition~\ref{P2} it is possible to show that multiple stationary solutions indeed exist for (some) small values of $\lambda$. Recall that the total energy $\mathcal{E}$ is non-coercive and a plain minimization thereof is thus not effective. However, when $\beta>0$ one can find additional critical points of $\mathcal{E}$ by considering a constrained minimization problem, where the idea is to minimize the mechanical energy $\mathcal{E}_m$ on a set of deflections with constant electrostatic energy $\mathcal{E}_e$.

%%%%%%%%%%%%%%%%%%%%%%%
\begin{theorem}[Multiplicity \cite{LW_Multi}]\label{T3}
Let $\beta>0$ and $\tau\ge 0$. Given \mbox{$\rho>2$}, there is $(\lambda_\rho,u_\rho,\psi_\rho)$ with $\lambda_\rho>0$, $u_\rho\in H_D^4(I)$, and $\psi_{u_\rho}\in H^2(\Omega(u_\rho))$ such that $(u_\rho,\psi_{u_\rho})$ is a stationary solution to \eqref{psi1d}-\eqref{bcu1d} with $\lambda=\lambda_ \rho$. Both $u_\rho=u_\rho(x)$ and $\psi_{u_\rho} = \psi_{u_\rho}(x,z)$ are even with respect to $x\in I$ and $-1<u_\rho<0$ in $I$. Moreover, $\lambda_\rho\rightarrow 0$ as $\rho\rightarrow\infty$ and $u_\rho\not= U_{\lambda_\rho}$ for all $\rho>2$ sufficiently large, the functions $U_\lambda$ being defined in Theorem~\ref{T1}.
\end{theorem}
%%%%%%%%%%%%%%%%%%%%%%%

\begin{proof}
For a given $\rho> 2$ define the (non-empty) set
\begin{equation*}
\mathcal{A}_\rho := \left\{ u \in H_D^2(I)\ ;\ u \;\text{ is even}\,, -1<u\le 0 \ ,\ \mathcal{E}_e(u)=\rho \right\} %\label{b2}
\end{equation*}
and 
\begin{equation*}
\mu(\rho) := \inf_{u\in \mathcal{A}_\rho} \mathcal{E}_m(u) \ge 0\,.
\end{equation*} 
Based on Proposition~\ref{P3} it can be shown that the function $\mu$ is non-decreasing on $(2,\infty)$ with
$$
\lim_{\varrho\to 2} \mu(\varrho)=0 \;\quad\text{ and }\;\quad \mu_\infty := \lim_{\varrho\to \infty} \mu(\varrho)  < \infty\ .
$$
Hence, taking a minimizing sequence $(u_k)_{k\ge 1}$ of $\mathcal{E}_m$ in $\mathcal{A}_\rho$ satisfying 
\begin{equation}\label{po}
\mu(\rho) \le \mathcal{E}_m(u_k) \le \mu(\rho) + \frac{1}{k} \,,
\end{equation}
one can combine \eqref{po} with the energy inequality \eqref{MarcusMiller} to show that there is a constant $\delta(\rho)\in (0,1)$ such that
\begin{equation}\label{Bobby}
-1+\delta(\rho) \le u_k(x) \le 0 \ , \quad x\in [-1,1]\ , \quad k\ge 1\,.
\end{equation}
Since the sequence $\big(\|\partial_x^2 u_k\|_{2}\big)_k$ is bounded due to \eqref{po}, one may extract a subsequence of $(u_k)_k$ which converges weakly to some $u_\rho$ in $H_D^2(I)$  still satisfying \eqref{Bobby}. Proposition~\ref{P2} entails $\mathcal{E}_e(u_\rho)=\rho$ so that $u_\rho$ is indeed a minimizer on $\mathcal{A}_\rho$ of $\mathcal{E}_m$ due to the latter's weak lower semi-continuity. Consequently, there is a Lagrange multiplier $\lambda_\rho\in\R$ such that
$$
\partial_u \mathcal{E}_m(u_\rho)-\lambda_\rho \partial_u \mathcal{E}_e(u_\rho)=0
$$
which, combined with classical elliptic regularity theory and Proposition~\ref{P2}, implies that $u_\rho\in H_D^4(I)$ is a classical solution to \eqref{u1ddstat}. Finally, owing to \eqref{b19} and the properties of $u_\rho$ one may derive a bound for the Lagrange multiplier $\lambda_\rho$ of the form
$$
0<\lambda_\rho\le \frac{c(\mu_\infty)}{(\rho-2)^2}
$$
which shows that $\lambda_\rho\rightarrow 0$ as $\mathcal{E}_e(u_\rho)=\rho\rightarrow\infty$ . In particular, $u_\rho\not= U_{\lambda_\rho}$ for all $\rho>2$ sufficiently large, since $\mathcal{E}_e(U_\lambda)\rightarrow 2$ as $\lambda\rightarrow 0$. 
\end{proof}

Theorem~\ref{T3} now yields multiplicity of stationary solutions for certain small values of~$\lambda$. Indeed, there is at least a sequence $\lambda_j\rightarrow 0$ of voltage values for which there are two different solutions $(u_j,\psi_{u_j})$ (i.e. $\rho=j$ in Theorem~\ref{T3}) and $(U_{\lambda_j}, \Psi_{U_{\lambda_j}})$ (i.e. $\lambda=\lambda_j$ in Theorem~\ref{T1}). Note that, by taking a different sequence $\rho_j\rightarrow\infty$ with $\rho_j\not= j$, one obtains different solutions $(u_{\rho_j},\psi_{u_{\rho_j}})$ but with possibly equal voltage values. We conjecture that the solutions constructed in Theorem~\ref{T3} actually lie on a smooth curve. This is supported by the numerical simulations performed in \cite{FlSxx} as well as by the results known for  the simplified vanishing aspect ratio model
\begin{equation}
\beta \partial_x^4 u_0 - \tau \partial_x^2 u_0 = - \frac{\lambda}{(1+u_0)^2}\ , \qquad x\in I\ , \label{u1SGstat}
\end{equation} 
supplemented with the boundary conditions \eqref{bcu1d}, 
see Section~\ref{Sec4.2} below. It is worth pointing out that the transformation for $\psi$ to a fixed domain introduced in Section~\ref{Sec3.1} is also used to perform the numerical simulations in \cite{FlSxx}.

A multiplicity result as stated in Theorem~\ref{T3} is not known up to now for the second-order case $\beta=0$ (but $\tau>0$), the main reason being the weaker regularity that comes along which appears to be insufficient to provide a control of the electrostatic energy $\mathcal{E}_e$ by the mechanical energy $\mathcal{E}_m$.

%%%%%%%%%%%%%%%%%%%%%%%
%%%%%%%%%%%%%%%%%%%%%%%
\subsubsection{Vanishing Aspect Ratio Limit}
%%%%%%%%%%%%%%%%%%%%%%%
%%%%%%%%%%%%%%%%%%%%%%%

Of course, the value $\lambda_s(\varepsilon)$ as well as the stationary solutions 
$$
(U_\lambda,\Psi_{U_\lambda})_{\lambda\in (0,\lambda_s(\varepsilon))} \;\;\text{ and }\;\; (u_\rho,\psi_{u_\rho})_{\rho>2}
$$ 
derived in Theorem~\ref{T1} and Theorem~\ref{T3}, respectively, depend on the aspect ratio $\ve$. It is then of primary interest to investigate the behavior of these solutions in the limit $\varepsilon\to 0$ and show that this limiting process produces stationary solutions to the stationary vanishing aspect ratio model \eqref{u1SGstat}. In particular, this puts the derivation of this model on firm ground which has been performed formally up to now. Roughly speaking, given a family $(u_\varepsilon, \psi_{u_\varepsilon})_\varepsilon$ of stationary solutions to \eqref{psi1d}-\eqref{icu1d} with $\lambda=\lambda_\varepsilon$, the main steps of the proof are the following:
\begin{itemize}
\item[(i)] Derive positive upper and lower bounds on $(\lambda_\varepsilon)_\varepsilon$, which do not depend on $\varepsilon$ small enough.
\item[(ii)] Derive estimates on $(u_\varepsilon)_\varepsilon$ in a suitable Sobolev space, e.g. in $W_q^2(I)$ for some $q\ge 2$, together with a positive lower bound on $(1+u_\varepsilon)_\varepsilon$, again not depending on $\varepsilon$ small enough.
\item[(iii)] Estimate the difference between $\psi_{u_\varepsilon}$ and $(x,z)\mapsto (1+z)/(1+u_\varepsilon(x))$ in a Sobolev space by positive powers of $\varepsilon$, which, in turn, provides an estimate of $g_\varepsilon(u_\varepsilon)-(1+u_\varepsilon)^{-2}$ and shows that it converges to zero as $\varepsilon\to 0$.
\end{itemize}
Once these three properties are established, a compactness argument is used to obtain the following convergence result:

%%%%%%%%%%%%%%%%%%%%%%%
\begin{theorem}[Vanishing Aspect Ratio Limit \cite{LW13, LW_v4}]\label{B}
There are a sequence $\varepsilon_k\to 0$, $\lambda_0>0$, and a (smooth) stationary solution $u_0$ to the stationary vanishing aspect ratio equation \eqref{u1SGstat} with $\lambda=\lambda_0$ such that
$$ 
\lim_{k\to\infty} \left[ \left| \lambda_{\varepsilon_k} - \lambda_0 \right| + \left\| u_{\varepsilon_k} - u_0\right\|_{ W_\infty^1(I)} \right] = 0
$$
and
\begin{equation}\label{411} 
\lim_{k\to \infty} \left[ \int_{\Omega(u_{\varepsilon_k})} \left| \psi_{u_{\varepsilon_k}}(x,z) - \frac{1+z}{1+u_{\varepsilon_k}(x)} \right|^2 \rd(x,z) + \left\| g_{\varepsilon_k}(u_{\varepsilon_k}) - \frac{1}{(1+u_{\varepsilon_k})^2} \right\|_2 \right]= 0\ .
\end{equation}
\end{theorem}
%%%%%%%%%%%%%%%%%%%%%%%

Let us emphasize that Theorem~\ref{B} above is valid either when $\beta>0$ \cite{LW_v4} or when $\beta=0$ and $\tau>0$ \cite{LW13}. However, the starting point of the analysis and actually the derivation of properties~(i)--(iii) listed above differ in both cases. In the second-order case $\beta=0$ stationary solutions for small voltage values can also be constructed by means of Schauder's fixed point theorem \cite{LW13} and this approach guarantees in particular a lower bound $\lambda_{min}>0$ for $\lambda_s(\varepsilon)$ independent of $\ve\in (0,1)$. This thus gives a family of stationary solutions $(u_\varepsilon,\psi_{u_\varepsilon})_{\varepsilon\in (0,1]}$ for each $\lambda\in (0,\lambda_{min}]$. Moreover, it yields {\it a priori} bounds on the stationary solutions of the form
\begin{equation}\label{bound2}
-1+\kappa_0\le u_\ve(x)\le 0\ ,\quad x\in I\ ,\qquad \| u_\ve\|_{W_\infty^2}\le \frac{1}{\kappa_0}
\end{equation}
for each $\lambda\in (0,\lambda_{min}]$, where $\kappa_0\in (0,1)$ is independent of $\ve\in (0,1)$. In that case, the family $(\lambda_\varepsilon)_\varepsilon$ in Theorem~\ref{B} is given by $\lambda_\varepsilon=\lambda\in (0,\lambda_{min}]$ and the properties~(i)--(ii) are satisfied. In the fourth-order case $\beta>0$, the outcome of Theorem~\ref{T3} is the starting point. Given $\rho>2$ and $\varepsilon>0$, we set $(\lambda_\varepsilon,u_\varepsilon) := (\lambda_\rho^\ve,u_\rho^\ve)$, the latter being constructed in Theorem~\ref{T3} and depending on $\varepsilon$, and use the variational construction of $u_\varepsilon$ to derive the uniform estimates~(i)--(ii). Property~(iii) follows from \eqref{psi1d} after multiplication by $\psi_{u_\varepsilon}$ and $\partial_z \psi_{u_\varepsilon}$. Note that the regularity of $\psi_{u_\varepsilon}$ is no longer the same in the $x$- and $z$-directions in the limit $\varepsilon\to 0$ since~\eqref{psi1d} becomes degenerate elliptic. Hence, a cornerstone of the proof is to obtain estimates for the trace of $\partial_z \psi_{u_\varepsilon}$ on the upper boundary of $\Omega(u_\ve)$ \cite{LW13, LW_v4}.

%%%%%%%%%%%%%%%%%%%%%%%
%%%%%%%%%%%%%%%%%%%%%%%
\subsection{Stationary Solutions to the Vanishing Aspect Ratio Model}\label{Sec4.2}
%%%%%%%%%%%%%%%%%%%%%%%
%%%%%%%%%%%%%%%%%%%%%%%

The study of stationary solutions to the vanishing aspect ratio model~\eqref{uSG} has a long history, starting with the seminal work \cite{JoL72} on the number of solutions to a broad class of semilinear Dirichlet problems in radial symmetry and then continued to cover second-order MEMS equations 
\begin{equation}\label{SGNoBending}
\begin{split}
-\tau \Delta u =-\,  \frac{\lambda}{(1+u)^2}\ \quad\text{in }\ D \ , \quad u=0 \quad\text{ on }\quad \partial D
\end{split}
\end{equation}
with open domain $D\subset\R^d$ (here $d$ is no longer restricted to $1,2$) and various variants hereof,  including those listed in Section~\ref{Sec2.4}. As expected from physics, there is a critical value $\lambda_0^{stat}>0$ -- depending on the domain's geometry -- such that no stationary solution exists if $\lambda>\lambda_0^{stat}$ and at least one stationary solution exists for $\lambda\in (0,\lambda_0^{stat})$.  More precisely, owing to the monotonicity and convexity of the function $z\mapsto 1/(1+z)^{2}$, an iterative monotone scheme can be set up to construct a branch of stationary solutions $\{ U_\lambda\, ;\, \lambda\in [0,\lambda_0^{stat}) \}$ to \eqref{SGNoBending} with $U_0\equiv 0$ and $U_{\lambda'} < U_\lambda$ in $I$ for all $0\le \lambda < \lambda'<\lambda_0^{stat}$. A consequence of this approach is that, given $\lambda\in (0,\lambda_0^{stat})$, each stationary solution $U_\lambda$ is also a maximal and stable solution in the following sense: Any other solution to \eqref{SGNoBending} with the same value of $\lambda$ lies below $U_\lambda$, and the principal eigenvalue of the associated linearized operator $-\Delta - 2\lambda/(1+U_\lambda)^3$ in $L_2(D)$ is positive. The behavior of $U_\lambda$ as $\lambda\to\lambda_0^{stat}$ as well as a further and more precise characterization of the bifurcation diagram depend strongly on the dimension~$d$. Concerning the former, owing to the monotonicity of the family $\{ U_\lambda\, ;\, \lambda\in [0,\lambda_0^{stat}) \}$ with respect to $\lambda$ and the lower bound $U_\lambda > -1$, the function
$$
U_{\lambda_0^{stat}} := \inf_{\lambda\in (0,\lambda_0^{stat})} U_\lambda \ge -1
$$
is always well-defined but might reach the value $-1$ at some points. However, this cannot occur for $d\in [1,7]$, and $U_{\lambda_0^{stat}}$ is then a classical solution to \eqref{SGNoBending} with $\lambda=\lambda_0^{stat}$, see \cite{GhG06}. If $d\ge 8$, then touchdown points may exist in the sense that $\min U_{\lambda_0^{stat}}=-1$: this is in particular the case when $D$ is the unit ball $\mathbb{B}_1$ of $\R^d$, see \cite{GhG06}. As for the bifurcation diagram, there are exactly two stationary solutions for each $\lambda\in (0,\lambda_0^{stat})$  if $d=1$, see~\cite{JoL72}. For dimensions $d\in [2,7]$ there is a curve $s\mapsto (\Lambda(s),U(s))$ of stationary solutions starting for $s=0$ at $(0,0)$ with infinitely many bifurcation or turning points so that there exists a unique (under additional geometric restrictions on $D$) stationary solution for sufficiently small voltage values~$\lambda$, exactly two stationary solutions for $\lambda$ in a small left neighborhood of $\lambda_0^{stat}$,  one extremal stationary solution at $\lambda=\lambda_0^{stat}$, and multiple stationary solutions for some values of $\lambda$ strictly between zero and $\lambda_0^{stat}$, see \cite{EsG08, EGG07, EGG10}. Additional properties of stationary solutions were investigated in dependence on the dimension of the ambient space and the geometry of the device. Proofs (and numerical evidence) for these results as well as regarding e.g. to estimates on the pull-in voltage or the {\it pull-in distance} (that is, the minimum of $1+U_{\lambda_0^{stat}}$) are to be found in \cite{BGP00, PeC03, PeT01, Pel01a, GPW05, FeZ05, GhG06, EsG08, EGG07, BeP08, CoG10, GuW08c}, see also \cite{EGG10} for an extended review. It is worth mentioning that the comparison principle is available in this case and  a key tool for the analysis in most of the papers cited above. 

\medskip

Less attention has so far been dedicated to the fourth-order problem
\begin{align}
\beta\Delta^2 u -\tau \Delta u =-\,  \frac{\lambda}{(1+u)^2}\ \quad\text{in }\ D \label{SGBending}\\
u=\beta\partial_\nu u=0\ \text{ on }\ \partial D\, . \label{SGBendingBC}
\end{align}
with $\beta>0$. The only result concerning \eqref{SGBending}-\eqref{SGBendingBC} in a general domain $D$ --~hence not relying on the maximum principle~-- seems to be \cite{LiY07}, where stationary solutions have been constructed for small values of~$\lambda$ by Banach's fixed point theorem. 

When restricting to the unit ball $\mathbb{B}_1$ of $\mathbb{R}^d$, $d\ge 1$, so that one may take advantage of the maximum principle in radial symmetry as stated in Proposition~\ref{P4}, similar arguments as in the second-order case provides a threshold value $\lambda_0^{stat}\in (0,\infty)$ such that, for $\lambda>\lambda_0^{stat}$, there is no stationary solution if $\tau=0$ and no radially symmetric stationary solution if $\tau>0$. Furthermore, there is a monotone branch of stable and maximal radially symmetric stationary solutions $\{U_\lambda\, ;\, \lambda\in [0,\lambda_0^{stat})\}$ with $U_0=0$,  see \cite{COG09} for $\tau=0$ and \cite{Boggio, LW14b} for $\tau>0$. The function $U_{\lambda_0^{stat}} := \inf_{\lambda\in (0,\lambda_0^{stat})} U_\lambda$ is again well-defined and its properties are investigated in \cite{CEGM10} when $\tau=0$. In particular, it is shown that the minimum of  $U_{\lambda_0^{stat}}$ in $\mathbb{B}_1$ is above $-1$ for $d\le 8$ so that $U_{\lambda_0^{stat}}$ is a smooth solution to \eqref{SGBending}-\eqref{SGBendingBC} in that case. In contrast, the minimum of  $U_{\lambda_0^{stat}}$ in $\mathbb{B}_1$ is equal to $-1$ if $d\ge 9$. Additional information on the structure of the solution set is provided in \cite{DFG10} for $d\ge 3$.

In dimensions $d=1,2$ more precise information on the bifurcation diagram is available, even when including  $\tau>0$:

%%%%%%%%%%%%%%%%%%%%%%%
\begin{theorem}[Bifurcation Diagram \cite{LW14b}]\label{TSSIntroduction}
Let $\beta>0$, $\tau\ge 0$, and consider $D=\mathbb{B}_1$ with $d\in \{1,2\}$. There exists $\lambda_0^{stat}>0$ such that there is no radially symmetric solution to \eqref{SGBending}-\eqref{SGBendingBC} for $\lambda>\lambda_0^{stat}$, at least one radially symmetric solution for $\lambda=\lambda_0^{stat}$, and at least two radially symmetric solutions for $\lambda\in (0,\lambda_0^{stat})$, a stable and an unstable one. These solutions together with the corresponding values  of $\lambda$ lie on an analytic curve and, as $\lambda$ approaches zero, the unstable solution converges to an explicitly computable function $\omega$ with minimum equal to $-1$.
\end{theorem}
%%%%%%%%%%%%%%%%%%%%%%%

Theorem~\ref{TSSIntroduction} for $\tau\ge 0$ gives a positive answer to the question of existence stated in \cite{EGG10} of at least two solutions for each $\lambda\in (0,\lambda_0^{stat})$ in dimension $d=2$. When $\tau=0$ the existence of at least two solutions for each $\lambda\in (0,\lambda_0^{stat})$ is also shown in \cite{GLY14} in the two-dimensional case $d=2$. A second solution different from the stable one is constructed therein by the Mountain Pass Principle. Its behavior as $\lambda\rightarrow 0$  is studied and shown to coincide with the one described above. Let us also point out that the asymptotic shape $\omega$ was found in \cite{LiW11} but by formal asymptotic expansions.

Theorem~\ref{TSSIntroduction} relies on the construction of an analytic curve $s\mapsto (\Lambda(s), U(s))$ in $\R\times C^4(\mathbb{B}_1)$ defined for $s\in [0,\infty)$ such that $U(s)$ is for each $s\in [0,\infty)$ a radially symmetric solution to \eqref{SGBending}-\eqref{SGBendingBC} with $\lambda=\Lambda(s)$. The curve emanates from  $(\Lambda(0), U(0))=(0, 0)$ with $(\Lambda(s), U(s))\rightarrow (0,\omega)$ as $s\rightarrow \infty$, where the end point $\omega$ is given as a solution of a boundary value problem in $\mathbb{B}_1\setminus\{0\}$ which can be computed explicitly. A plot of $\omega$ is shown in Figure~\ref{fig1}, its qualitative behavior being the same for $d=1$ and $d=2$. 
%%%%%%%%%%%%%%%%%%%%%%%
\begin{figure}[h]
\centering\includegraphics[width=6cm]{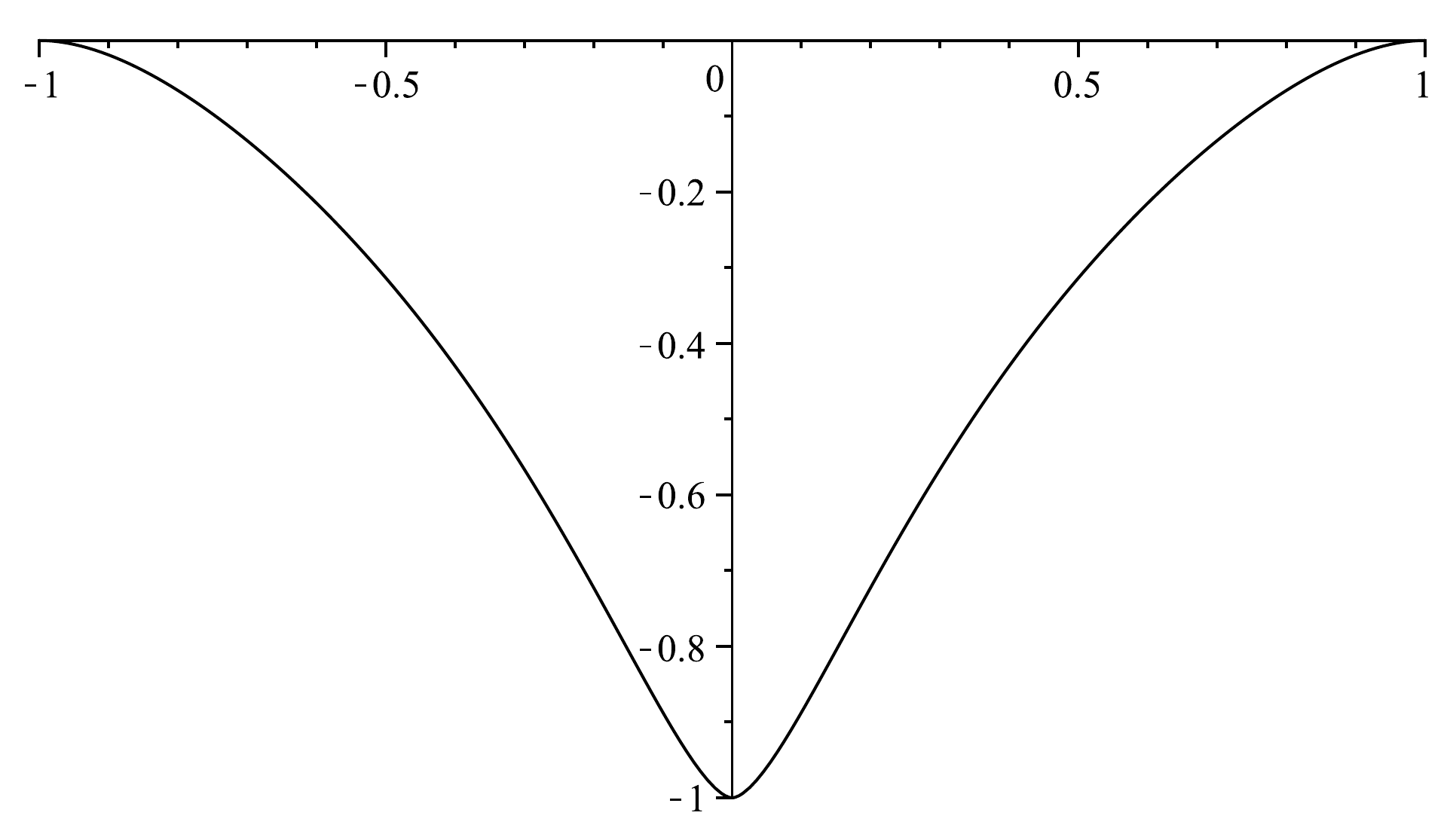}
%\centering\includegraphics[width=6cm]{Final2Dim(1,1).eps}
\caption{\small Plot of the endpoint $\omega$ of the curve $(\Lambda(s), U(s))$ for $d=2$ and $(\beta,\tau)=(1,50)$.}\label{fig1}
\end{figure}
%%%%%%%%%%%%%%%%%%%%%%%
In addition, there is $s_*>0$ such that $\Lambda$ is an increasing function from $[0,s_*]$ onto $[0,\lambda_0^{stat}]$ and $\Lambda$ is decreasing in a right-neighborhood of $s_*$ with $\Lambda(s)\rightarrow 0$ as $s\rightarrow \infty$. Consequently, given $\lambda\in (0,\lambda_0^{stat})$ there are indeed two values $0<s_1<s_*<s_2$ with $\Lambda(s_j)=\lambda$ for $j=1,2$ and $U(s_2)\le U(s_1)$ in $\mathbb{B}_1$ with $U(s_2)\not= U(s_1)$.

The proof of Theorem~\ref{TSSIntroduction} is performed in several steps. The implicit function theorem provides a smooth branch $\mathcal{A}_0$ of radially symmetric solutions $(\lambda,u)$ to~\eqref{SGBending}-\eqref{SGBendingBC} emanating from $(0,0)$ which can be extended to a global curve $\mathcal{A}$ by means of bifurcation theory for real analytic functions \cite{BDT00}, a tool that was also used in \cite[Sect.~6.2]{EGG10} for the second-order case $\beta=0$. It can then be shown  that the branch $\{ (\lambda,U_\lambda)\, ;\, \lambda\in [0,\lambda_0^{stat})\}$ of stable and maximal solutions to \eqref{SGBending}-\eqref{SGBendingBC} constructed by the iterative monotone scheme described above actually coincides with $\mathcal{A}_0$. {\it A priori} estimates and the fact that $U_\lambda$ is decreasing in $\lambda$ entail that a radially symmetric stationary solution $U_{\lambda_0^{stat}}$ exists also at $\lambda=\lambda_0^{stat}$, which guarantees, on the one hand, that $\mathcal{A}_0\not=\mathcal{A}$ and, on the other hand, that the bifurcation result of \cite{CrR73} applies to extend the branch $\mathcal{A}_0$ beyond $(\lambda_0^{stat}, U_{\lambda_0^{stat}})$. The final and most intricate step is to show that~$\mathcal{A}$ connects  $(0,0)$ to the end point $(0,\omega)$ and to identify the latter. As a consequence, the continuous curve $\mathcal{A}$ passes through  $(\lambda_0^{stat},U_{\lambda_0^{stat}})$.

\begin{rems}
\begin{enumerate}
\item When $\tau>0$ the maximum principle is only valid for radially symmetric functions by Proposition~\ref{P4} and we can only guarantee the non-existence of radially symmetric stationary solutions to \eqref{SGBending}-\eqref{SGBendingBC} for $\lambda>\lambda_0^{stat}$. It is possible (though unlikely) that non-radially symmetric stationary solution exist above this threshold value.
\item The qualitative bifurcation diagram is expected to be completely different in higher space dimensions $d\ge 3$ as pointed out in \cite{DFG10, EGG10}.
\end{enumerate}
\end{rems}

For numerical results on \eqref{SGBending}-\eqref{SGBendingBC} and investigations based on formal asymptotic methods we refer to \cite{LiW11, LiW08, KLW11}. In particular, by means of a shooting argument one can numerically verify that \eqref{SGBending}-\eqref{SGBendingBC} admits no other solutions than found in Theorem~\ref{TSSIntroduction} \cite{LiW08}.

\medskip

We end this section with a few words about \eqref{SGBending} when it is supplemented with pinned boundary conditions \eqref{bcubcuPinned}, that is,
\begin{equation}
\beta\Delta^2 u -\tau \Delta u =-\,  \frac{\lambda}{(1+u)^2}\ \quad\text{in }\ D\ , \quad u=\beta\Delta u=0\ \text{ on }\ \partial D\, , \label{Bach}
\end{equation}
with $\beta>0$ and $\tau\ge 0$. Interestingly, the maximum principle  holds in arbitrary smooth domains \cite{EGG10, GGS10} and allows one at least in part to show similar results for the fourth-order problem as outlined above for the second-order case. In particular, there is again $\lambda_0^{stat}>0$ such that no solution to \eqref{Bach} exists for $\lambda>\lambda_0^{stat}$ while there is a branch of stable maximal solutions $\{U_\lambda\,;\,\lambda\in (0,\lambda_0^{stat})\}$ obtained from a monotone iterative scheme such that $U_\lambda$ is  decreasing in~$\lambda$, see \cite{LiY07}. If $d\le 3$, then the minimum of $U_{\lambda_0^{stat}} := \inf_{\lambda\in (0,\lambda_0^{stat})} U_\lambda$ is greater than $-1$ and $U_{\lambda_0^{stat}}$ is the unique classical solution to \eqref{SGBending} subject to \eqref{bcubcuPinned}, see \cite{GuW08a}. A similar result is available in \cite{CEG10, CoG14} when $\tau=0$ and $d\le 6$. In the particular case of $D$ being a ball, the extremal solution is still regular in the sense that $\inf_{D}U_{\lambda_0^{stat}} >-1$  for dimension $d\le 8$ \cite{CEGM10, Mor10} while a touchdown in the sense of $\inf_{D}U_{\lambda_0^{stat}} =-1$ occurs for $d\ge 9$ (provided $\tau/\beta$ is small) \cite{Mor10}. If $D$ is a two-dimensional convex domain, then there is a second stationary solution for $\lambda\in (0,\lambda_0^{stat})$ constructed by the Mountain Pass Principle that approaches $-1$ as $\lambda\rightarrow 0$ \cite{GuW08a, GLY14}. If $D$ is the three-dimensional unit ball, then \eqref{Bach} with $\tau=0$ has infinitely many fold points for the bifurcation branch corresponding to radially symmetric solutions \cite{GuW08b}. For more precise and further results for \eqref{Bach} we refer to the abovementioned references \cite{GuL11, GuW08a, EGG10, LiY07, CEG10, GuW08b, Mor10, CEGM10, CoG14, GLY14} as well as to \cite{Nee09, CFT11}.

%%%%%%%%%%%%%%%%%%%%%%%
%%%%%%%%%%%%%%%%%%%%%%%
\section{Evolution Problem}\label{S5}
%%%%%%%%%%%%%%%%%%%%%%%
%%%%%%%%%%%%%%%%%%%%%%%

We next turn to the dynamical free boundary MEMS equations~\eqref{psi1d}-\eqref{icu1d}. Let us recall that it is expected on physical grounds that there is a (dynamical) pull-in voltage value determining the dynamics of the device in the following sense: Solutions corresponding to values of $\lambda$ below this threshold are global in time, while for values above the threshold solutions cease to exist globally as the elastic plate pulls in at some finite time $T_m<\infty$, i.e., 
\begin{equation}
\lim_{t\to T_m} \big(\min_{x\in I}\, u(t,x) \big) = -1\ . \label{PLtouchdown}
\end{equation}
Rigorous proofs for such a  qualitative behavior  with a sharp threshold are obviously non-trivial, particularly for the free boundary problem \eqref{psi1d}-\eqref{icu1d}. Nevertheless, some of the described behavior is established. As we shall see in the next sections, the fourth-order case $\beta>0$ and the second-order case $\beta=0$ are somewhat different with respect to what is proven so far for large voltage values. In the former case, non-global solutions necessarily exhibit the touchdown behavior described in \eqref{PLtouchdown} but it is not known (though expected) that such solutions indeed exist for large $\lambda$. In the latter case, solutions cannot exist globally in time for large $\lambda$ but may exhibit also an alternative singularity in form of a norm blow up.
Roughly speaking, the reason for this difference is that the fourth-order case is technically more involved, but comes along with a higher regularity. 

We first consider in Section~\ref{Sec5.1} the local well-posedness and then in Sections~\ref{Sec5.1.5}-\ref{Sec5.2.5} questions related to global existence. The dependence of the solutions on some parameters is then considered: in Section~\ref{Sec5.3} the damping dominated limit $\gamma\rightarrow 0$ is studied and in Section~\ref{Sec5.4} the vanishing aspect ratio limit $\ve\rightarrow 0$.

%%%%%%%%%%%%%%%%%%%%%%%
%%%%%%%%%%%%%%%%%%%%%%%
\subsection{Local Existence}\label{Sec5.1}
%%%%%%%%%%%%%%%%%%%%%%%
%%%%%%%%%%%%%%%%%%%%%%%

We recall that one can formulate~\eqref{psi1d}-\eqref{icu1d} as a single  nonlocal problem of the form~\eqref{u1dd}  only involving $u$, the corresponding electrostatic potential $\psi=\psi_u$ being determined with the help of Proposition~\ref{P1}. Specifically, recalling that $g_\ve(u)$ is the square of the trace of the (rescaled) gradient of the solution $\psi_u$ to \eqref{psi1d}-\eqref{bcpsi1d} on the upper boundary of $\Omega(u)$, see \eqref{gg}, the function $u$ solves
\begin{equation}\label{rapu1d}
\gamma^2\partial_t^2 u+\partial_t u +\beta\partial_x^4 u- \tau \partial_x^2 u =-\lambda\, g_\ve(u)\ , \quad t>0\ , \quad x\in I
\end{equation}
with clamped boundary conditions
\begin{align}
u(t,\pm 1) &=\beta\partial_x u(t,\pm 1)=0\ , \quad t>0 \ , \label{rapbcu1d}
\end{align}
and zero initial conditions
\begin{align}
u(0,x)&= \gamma \partial_tu(0,x)=0\ ,\qquad x\in I\ ,\label{rapicu1d}
\end{align}
the latter only assumed for a simpler exposition. To discuss the solvability of \eqref{rapu1d}-\eqref{rapicu1d} we fix $2\alpha \in (0,1/2)$ and introduce the operator 
$$
A_\alpha :=\beta \partial_x^4 - \tau \partial_x^2\in\mathcal{L}\big(H_D^{2\ell+2\alpha}(I),H_D^{2\alpha}(I)\big)
$$ 
with $\ell=2$ if $\beta>0$ and $\ell=1$ if $\beta=0$ and $\tau>0$. 

\medskip

Let us first consider the case $\beta>0$ and $\gamma>0$. We reformulate~\eqref{rapu1d}-\eqref{rapicu1d} as a hyperbolic nonlocal semilinear Cauchy problem for  ${\bf u}=(u,\gamma\partial_tu)^T$ of the form
\begin{equation}\label{CPhyp}
\gamma {\bf u}'+\mathbb{A}_\alpha {\bf u}= -\lambda f({\bf u})\ ,\quad t>0\ ,\qquad {\bf u}(0)={\bf 0}
\end{equation}
with ${\bf u}'$ indicating the time derivative of ${\bf u}$. Hereby, the matrix operator
$$
\mathbb{A}_\alpha:=\left(
\begin{matrix} 
0 & -1\\
 & \\
A_\alpha &  1/\gamma
\end{matrix}
\right)
$$
with domain  $D(\mathbb{A}_\alpha):=H_D^{4+2\alpha}(I)\times H_D^{2+2\alpha}(I)$ generates a strongly continuous group $e^{-t\mathbb{A}_\alpha}$, $t\in\R$, on the Hilbert space $\mathbb{H}_\alpha:=H_D^{2+2\alpha}(I)\times H_D^{2\alpha}(I)$ with an exponential decay, the latter being due to the damping term $\partial_t u$ in \eqref{rapu1d} (e.g. see the cosine function theory in \cite[Sections~5.5 \&~5.6]{Are04}). Since $H_D^{2+2\alpha}(I)$ is continuously embedded in $W_{q,D}^2(I)$ for some $q>2$ and $\alpha\in (0,1/4)$, Proposition~\ref{P1} entails that the function
$$
f:S_2^{2+2\alpha}(\kappa)\times H_D^{2\alpha}(I)\rightarrow \mathbb{H}_\alpha
$$ 
with
\begin{equation*}\label{f1}
f({\bf u}):=\left(
\begin{matrix} 
0\\ 
 g_\ve(u)
\end{matrix}\right)\,,\quad {\bf u}\in S_2^{2+2\alpha}(\kappa)\times H_D^{2\alpha}(I)\,,
\end{equation*}
is bounded and uniformly Lipschitz continuous when $\kappa\in (0,1)$. A classical application of the contraction mapping principle then yields a unique strong solution ${\bf u}\in C([0,T],\mathbb{H}_\alpha)$ to~\eqref{CPhyp} with ${\bf u}'\in L_{1}(0,T;\mathbb{H}_\alpha)$ on a (possibly small) time interval~$[0,T]$, which can be continued as long as $\|{\bf u}(t)\|_{\mathbb{H}_\alpha}$ stays bounded and $u(t)>-1$. This in particular yields a criterion for global existence. 

Let us point out that, unfortunately, this approach does not work in the second-order case $\beta=0$ with $\gamma>0$ in \eqref{rapu1d}-\eqref{rapicu1d} since the right-hand side $g_\ve$ is -- roughly speaking -- of order slightly above one (see Proposition~\ref{P1}) and thus exceeds half of the operator's order.

\medskip

The just described difficulty does not occur if \mbox{$\gamma=0$} due to parabolic smoothing effects. Indeed, in this case we may consider $\beta>0$ and $\beta=0$ (with $\tau>0$ then) and handle equation \eqref{rapu1d}-\eqref{rapicu1d} as a parabolic nonlocal Cauchy problem in $H_D^{2\alpha}(I)$ of the form
$$
u' +A_\alpha u=-\lambda g_\ve(u)\,,\quad t>0\, ,\qquad u(0)=0
$$
possessing a classical solution in $C\big([0,T_m),H_D^{2\ell+2\alpha}(I)\big)\cap C^1\big([0,T_m),H_D^{2\alpha}(I)\big)$. 

In summary, we have the following result regarding the local well-posedness of \eqref{rapu1d}-\eqref{rapicu1d}:

%%%%%%%%%%%%%%%%%%%%%%%
\begin{theorem}[Local Existence \cite{LW14a,ELW14}]\label{T6}
Assume that either $\beta>0$, $\gamma\ge 0$, $\tau\ge 0$ or $\beta=\gamma = 0$, $\tau>0$ and consider $\lambda>0$.

\begin{itemize}
\item[(i)] There is a unique (smooth) solution $(u,\psi_u)$ to \eqref{rapu1d}-\eqref{rapicu1d} on a maximal interval of existence $[0,T_{m})$ such that
$$
u(t,x)>-1\ ,\quad (t,x)\in [0,T_m)\times I\ . 
$$ 

\item[(ii)] If the solution does not exist globally in time, that is, if $T_m<\infty$, then a touchdown singularity \eqref{PLtouchdown} occurs or some Sobolev norm of $\big(u(t),\gamma\partial_t u(t)\big)$ blows up as \mbox{$t\rightarrow T_m$}.
\end{itemize}
\end{theorem}
%%%%%%%%%%%%%%%%%%%%%%%

As pointed out above, the regularity of the solution depends on the particular case one considers, that is, on whether or not the parameters $\beta$ and $\gamma$ are zero. For further use in Sections~\ref{Sec5.3} and~\ref{Sec5.4} let us emphasize that the solution $(u,\psi_u)$ as well as the maximal existence time $T_m$ depend on both values $\ve>0$ and $\gamma\ge 0$.

%%%%%%%%%%%%%%%%%%%%%%%
%%%%%%%%%%%%%%%%%%%%%%%
\subsection{Refined Criterion for Global Existence}\label{Sec5.1.5}
%%%%%%%%%%%%%%%%%%%%%%%
%%%%%%%%%%%%%%%%%%%%%%%

Clearly, the outcome of Theorem~\ref{T6}~(ii) is not yet satisfactory as it does not allow one to figure out whether the occurring finite time singularity is a touchdown {\it or} a norm blowup of $u(t)$ as  \mbox{$t\rightarrow T_m<\infty$}, the latter being though a non-observed phenomenon in real world applications. 

We shall thus address next the preclusion of a norm blowup for non-global solutions. This requires {\it a priori} estimates on solutions for which a promising source may be the gradient flow structure of \eqref{rapu1d}-\eqref{rapicu1d} pointed out in Section~\ref{Sec3.2}. However, when $\beta=0$ the best estimates one can hope for coming from the functional $\mathcal{E}$ in~\eqref{EE} is an $H^1(I)$-estimate on $u(t)$ (by definition of the mechanical energy $\mathcal{E}_m$) which seems to be too weak, unfortunately,  as can be seen from \eqref{MarcusMiller}. In contrast, the additional regularity stemming from the  fourth-order bending term $\beta\partial_x^4$ is more promising as it leads to sufficient $H^2(I)$-estimates on $u(t)$. The drawback when including fourth-order bending is, though, that no (dynamical) comparison principle is available as noticed in Section~\ref{Sec3.2} and so even a preliminary upper bound on $u$ is not obvious (again in contrast to the second-order case $\beta=\gamma=0$ where the comparison principle applied to {\eqref{rapu1d} immediately implies the bound  $u(t)\le 0$ since $g_\ve(u)>0$). A further difficulty encountered in deriving {\it a priori} estimates is the non-coercivity of the corresponding functional $\mathcal{E}$ from \eqref{EE}. We now give more details on how to overcome these issues.

The strategy to prove that necessarily the touchdown singularity \eqref{PLtouchdown} occurs in the fourth-order case $\beta>0$ when $T_m<\infty$ is the following: Supposing that there are $T>0$ and $\kappa\in (0,1)$ such that 
$$
u(t,x) \ge -1 + \kappa \;\;\text{ for }\;\; (t,x) \in ([0,T]\cap [0,T_{m})) \times I\ ,
$$
Theorem~\ref{T6}~(ii) implies $T_m\ge T$ provided one can control a suitable Sobolev norm of $(u(t), \gamma \partial_t u(t))$, see the discussion in Section~\ref{Sec5.1}. To find such a bound one first observes that -- even though there is no {\it a priori} upper bound on $u$ as pointed out above -- a weighted $L_1$-bound of the form 
\begin{equation*}
\int_{-1}^1 \zeta_1(x) |u(t,x)|\ \mathrm{d}x \le c_0\, ,\quad t\in [0,T_m)\,, \label{rc20.b}
\end{equation*}
for sufficiently small $\gamma\ge 0$ can be derived, where $\zeta_1$ is the positive eigenfunction of the operator $\beta \partial_x^4 - \tau \partial_x^2$  subject to clamped boundary conditions provided by Proposition~\ref{P4}. This estimate is obtained by testing equation~\eqref{rapu1d} against $\zeta_1$ and combining the outcome of the computation with the lower bound on $u$. Using the estimate~\eqref{MarcusMiller}  on the electrostatic energy $\mathcal{E}_e(u(t))$  together with an interpolation inequality and a variant of Poincar\'e's inequality, one deduces
\begin{equation*}
\begin{split}
\mathcal{E}_e(u(t)) &=\int_{\Omega(u)} \left[ \varepsilon^2 |\partial_x \psi_u(t,x,z)|^2 + |\partial_z \psi_u(t,x,z)|^2 \right]\ \mathrm{d}(x,z)\\
&\le \int_{-1}^1 \frac{1 + \varepsilon^2 |\partial_x u(t,x)|^2}{1+u(t,x)}\ \mathrm{d}x\\
&\le c(\kappa)\left( 1+ \delta \|\partial_x^2 u(t)\|_{2}^2 +c(\delta)\left(\int_{-1}^1 \zeta_1(x) |u(t,x)|\ \mathrm{d}x \right)^2\right)\ 
\end{split}
\end{equation*}
for $\delta\in (0,1)$ and $t\in [0,T]\cap [0,T_m)$. Recognizing part of the mechanical energy $\mathcal{E}_m(u(t))$ on the right-hand side of this inequality,  the weighted $L_1$-estimate on $u$ from above and Corollary~\ref{C1} yield (for sufficiently small~$\delta>0$) 
$$
\frac{1}{2} \mathcal{E}_m(u(t)) - c(\kappa) \le \mathcal{E}(u(t))
\le \mathcal{E}(0)\, , \quad t\in [0,T]\cap [0,T_m)\,, 
$$
hence
$$
\|u(t)\|_{H^2(I)} \le c(\kappa)\, , \quad t\in [0,T]\cap [0,T_m)\, .
$$
Applying Proposition~\ref{P1}~(ii) then gives a suitable bound on $g_\ve(u(t))$ and thus on the right-hand side of the Cauchy problem~\eqref{CPhyp}, what readily 
yields the desired {\it a priori} bound on $(u, \gamma \partial_t u)$. Consequently, if $T_m<\infty$, then \eqref{PLtouchdown} occurs. 

Summarizing, we have the following result regarding global existence for the fourth-order problem.

%%%%%%%%%%%%%%%%%%%%%%%
\begin{theorem}[Criterion for Global Existence when $\beta>0$ \cite{LW14a}]\label{T7}
Let $\beta>0$ and $\tau\ge 0$. There is $\gamma_0>0$ such that if $\gamma\in [0,\gamma_0)$, then $T_m<\infty$ implies the touchdown behavior \eqref{PLtouchdown}. 
\end{theorem}
%%%%%%%%%%%%%%%%%%%%%%%

%%%%%%%%%%%%%%%%%%%%%%%
%%%%%%%%%%%%%%%%%%%%%%%
\subsection{Global Existence}\label{Sec5.2}
%%%%%%%%%%%%%%%%%%%%%%%
%%%%%%%%%%%%%%%%%%%%%%%

An interesting byproduct of the fixed point argument leading to local existence is that it simultaneously yields global existence  for small voltage values $\lambda$ in any of the cases considered in Theorem~\ref{T6}.

%%%%%%%%%%%%%%%%%%%%%%%
\begin{theorem}[Global Existence \cite{ELW14, LW14a}]\label{T11}
Assume that either $\beta>0$, $\gamma\ge 0$, $\tau\ge 0$ or $\beta=\gamma = 0$, $\tau>0$. There is $\lambda_{ge}(\ve,\gamma)>0$ such that the solution to \eqref{rapu1d}-\eqref{rapicu1d} with maximal interval of existence $[0,T_{m})$ provided by Theorem~\ref{T6} is global for any $\lambda\in (0,\lambda_{ge}(\ve,\gamma))$, that is, $T_m=\infty$. In addition, $u$ is bounded in $H^{2+2\alpha}(I)$ uniformly with respect to time, where  $\alpha\in (0,1/4)$, and 
$$
\inf_{(t,x)\in [0,\infty)\times I} u(t,x)>-1\ .
$$
Finally, any cluster point in $H^{2+\alpha}(I)$ of $\{ u(t)\ ;\ t\ge 0\}$ as $t\to\infty$ is a stationary solution to \eqref{rapu1d}-\eqref{rapicu1d}.
\end{theorem}
%%%%%%%%%%%%%%%%%%%%%%%

The damping term $\partial_t u$ in \eqref{rapu1d} is crucial for the validity of Theorem~\ref{T11}. Indeed, the proof of the latter mainly relies on the exponential decay of the corresponding underlying semigroup  which entails that the contraction mapping principle used to establish Theorem~\ref{T6} can be applied on any (finite) time interval when $\lambda$ is small (see \cite{LW14a, ELW14} for details). Note that no touchdown occurs in case of Theorem~\ref{T11}, not even in infinite time. 

To show the last statement of Theorem~\ref{T11} observe that the bounds on $u$ stated in Theorem~\ref{T11} and \eqref{MarcusMiller} entail that $\mathcal{E}(u)$ is bounded uniformly with respect to time and that $\{ u(t)\ ;\ t\ge 0\}$ is relatively compact in $H^{2+\alpha}(I)$. Combining the bound on $\mathcal{E}(u)$ with \eqref{rc2b} shows that
$$
\int_0^\infty \|\partial_t u(t)\|_2^2\ \rd t < \infty\ .
$$
The compactness of  the trajectory entails that there is a sequence $t_n\rightarrow\infty$ and $w \in H^{2+\alpha}(I)$ such that $u(t_n)\longrightarrow w$ in $H^{2+\alpha}(I)$. Classical arguments of dynamical systems then guarantee that $(w,\psi_w)$ is a stationary solution to \eqref{rapu1d}-\eqref{rapicu1d}.
A further consequence of the last statement in Theorem~\ref{T11} is that $\lambda_{ge}(\ve,\gamma)\le \lambda_{ns}(\ve)$, the latter being defined in Theorem~\ref{T2}.

%%%%%%%%%%%%%%%%%%%%%%%
%%%%%%%%%%%%%%%%%%%%%%%
\subsection{Finite Time Singularity}\label{Sec5.2.5}
%%%%%%%%%%%%%%%%%%%%%%%
%%%%%%%%%%%%%%%%%%%%%%%

Showing the physically expected behavior that solutions to \eqref{rapu1d}-\eqref{rapicu1d} cannot exist globally in time for large voltage values is more delicate and actually known only in the  second-order case when $\beta=\gamma=0$. A classical technique \cite{Kap63} to prove that solutions to the vanishing aspect ratio model \eqref{uSG} only exist on a finite time interval for $\lambda$ large is to study the time evolution of 
$$
E_0(t) := \int_{-1}^1 \phi_1(x)\, u(t,x)\ \rd x\ , \quad t\in [0,T_{m})\ ,
$$
wher $\phi_1$ is a suitably normalized positive eigenfunction of the Laplace operator with homogeneous Dirichlet boundary conditions, and show that $E_0$ reaches $-1$ in finite time \cite{FMPS06, GPW05} which corresponds to the occurrence of a touchdown, see Section~\ref{Sec5.5} for more precise statements. As described in Section~\ref{Sec4.1}, a similar approach is used in \cite{LW13} for the stationary version of \eqref{rapu1d}-\eqref{rapicu1d} and heavily relies on the convexity of $u$ with respect to~$x$. Unfortunately, such a convexity property is not known for the evolution problem~\eqref{rapu1d}-\eqref{rapicu1d}  (neither it is for \eqref{uSG}), the crucial bound \eqref{convex} thus not being available. However, suitable integral estimates on the trace of the gradient of $\psi_u(\cdot, u)$ on the upper boundary of $\Omega(u)$ relating it to the behavior of $u$ allow one to study the time evolution of the nonlinear variant 
\begin{equation*}
E_\alpha(t) := \int_{-1}^1 \phi_1(x)\, \left( u + \frac{\alpha}{2}\ u^2\right)(t,x)\ \rd x\ , \quad t\in [0,T_{m})\ , 
\end{equation*}
for which one can derive by Jensens's inequality an energy inequality of the form
\begin{equation*}
\frac{\rd }{\rd t}E_\alpha(t)\le c_1+c_2(\lambda)\left(c_3(\lambda,\alpha)-\frac{1}{1+E_\alpha(t)}\right)\ , \quad t\in [0,T_{m})\ , 
\end{equation*}
which has no global solution when $\alpha\in (0,1)$ is suitably chosen and $\lambda$ is sufficiently large.  Summarizing, we obtain the following result:

%%%%%%%%%%%%%%%%%%%%%%%
\begin{theorem}[Finite Time Singularity \cite{ELW14}]\label{T12}
Let $\beta=\gamma=0$. There is $\lambda_{sing}(\ve)>0$ such that, if $\lambda>\lambda_{sing}(\ve)$, then the maximal interval of existence $[0,T_{m})$ of the solution to the second-order parabolic problem \eqref{rapu1d}-\eqref{rapicu1d} provided by Theorem~\ref{T6} is finite, that is, $T_{m}<\infty$. Clearly $\lambda_{sing}(\ve)\ge \lambda_{s}(\ve)$, the latter being defined in Theorem~\ref{T1}.
\end{theorem}
%%%%%%%%%%%%%%%%%%%%%%%

The above mentioned integral estimates on which the proof of Theorem~\ref{T12} relies are also at the heart of the study of non-existence of stationary solutions for the fourth-order free boundary problem as already outlined in Section~\ref{Sec4.1}.

%%%%%%%%%%%%%%%%%%%%%%%
%%%%%%%%%%%%%%%%%%%%%%%
\subsection{The Damping Dominated Limit $\gamma\rightarrow 0$}\label{Sec5.3}
%%%%%%%%%%%%%%%%%%%%%%%
%%%%%%%%%%%%%%%%%%%%%%%

When damping effects are assumed to be predominant, inertial effects are often neglected from the outset in the modeling. Most research papers hitherto are dedicated to the corresponding  ``parabolic'' version of \eqref{rapu1d}-\eqref{rapicu1d} or \eqref{uSG} with $\gamma=0$. We now discuss the behavior of the solutions to the ``hyperbolic'' problem \eqref{rapu1d}-\eqref{rapicu1d} when $\gamma>0$ in the damping dominated limit $\gamma\to 0$. The starting point is Theorem~\ref{T6} which provides local solutions to  \eqref{rapu1d}-\eqref{rapicu1d} in the fourth-order case $\beta>0$ likewise for $\gamma>0$ and $\gamma=0$. 

\medskip

Let us emphasize that not only the unique solution $(u,\psi_u)$ to  \eqref{rapu1d}-\eqref{rapicu1d} provided by Theorem~\ref{T6} depends on $\gamma\ge 0$, but also its maximal time of existence $T_{m}$, that is, $(u,\psi_u)=(u_\gamma,\psi_\gamma)$ and $T_{m}=T_{m,\gamma}$. Obviously, considering the time singular limit $\gamma\to 0$ requires a common interval of existence which is independent of $\gamma$, that is, a lower bound $T>0$ on the maximal existence time~$T_{m,\gamma}$. Based on this one can show that $(u_\gamma,\psi_\gamma)$ converges toward $(u_0,\psi_0)$ in a suitable sense as $\gamma\to 0$.

%%%%%%%%%%%%%%%%%%%%%%%
\begin{theorem}[Damping Dominated Limit \cite{LW_Singular}]\label{PLT}
Suppose the assumptions of Theorem~\ref{T6} with $\beta>0$ and let $(u_\gamma,\psi_\gamma)$ be the unique solution to \eqref{rapu1d}-\eqref{rapicu1d} on the maximal interval of existence $[0,T_{m,\gamma})$ provided by Theorem~\ref{T6} for $\gamma\ge 0$. There is $T>0$ such that
\begin{equation}\label{PLs1}
(u_\gamma,\Phi_{u_\gamma})\longrightarrow (u_0,\Phi_{u_0})\ \text{ in } \ C\big([0,T],H^{2}(I)\times H^{2}(\Omega))\big)\ \text{ as $\gamma\longrightarrow  0$\,,}
\end{equation}
 where $\Phi_{u_\gamma}$ is the transformed electrostatic potential introduced in \eqref{PLphi} (with $u$ replaced by $u_\gamma$).
If $\lambda>0$ is sufficiently small, then $T_{m,\gamma}=\infty$ for any $\gamma$ small enough and the statement \eqref{PLs1} is true for each $T>0$.
\end{theorem}
%%%%%%%%%%%%%%%%%%%%%%%

Deriving a lower bound for the maximal existence time independent of the parameter $\gamma$ is the main part of the proof of Theorem~\ref{PLT} and relies on an exponential decay of the energy associated with the damped wave equation (an idea inspired by \cite{HaZ88})  which is independent of $\gamma$ near zero. The energy functional $\mathcal{E}$ introduced in Section~\ref{Sec3.2} then again plays an important r\^ole as it provides compactness properties of $(u_\gamma,\Phi_{u_\gamma})$ needed in order to obtain a cluster point, its identification as $(u_0,\Phi_{u_0})$ then being due to parabolic singular perturbation theory outlined e.g. in \cite{Fat85, Kis63}. For further details we refer to~\cite{LW_Singular}.

%%%%%%%%%%%%%%%%%%%%%%%
%%%%%%%%%%%%%%%%%%%%%%%
\subsection{The Vanishing Aspect Ratio Limit $\ve\rightarrow 0$ }\label{Sec5.4}
%%%%%%%%%%%%%%%%%%%%%%%
%%%%%%%%%%%%%%%%%%%%%%%

A natural and interesting question concerns the relation between the vanishing aspect ratio equation \eqref{uSG} and the free boundary problem \eqref{rapu1d}-\eqref{rapicu1d}. In particular, whether the former can be obtained as a rigorous limit from the latter. In Theorem~\ref{B} we have seen that this is indeed the case for the stationary problem. We next show that this is true as well for the evolutionary problem and thus give a sound justification of the vanishing aspect ratio model. 

Let us recall that the solution $(u,\psi_u)$ to \eqref{rapu1d}-\eqref{rapicu1d} on the maximal interval of existence $[0,T_m)$ provided by Theorem~\ref{T6} depends on the value of the aspect ratio $\ve>0$ and we shall thus write $(u_\ve,\psi_\ve)$ and $T_m^\ve$ to indicate this. As in the previous study of the damping dominated limit, the starting and simultaneously most difficult point is the derivation of a lower bound $T_0>0$ on $T_m^\ve$ independent of $\ve>0$. This relies on an estimate on the transformed electrostatic potential $\Phi_{u_\ve}$ introduced in \eqref{PLphi}. Carefully exploiting the structure of equation~\eqref{PLPhi} allows one to control the $H^2(\Omega)$-norm of $\Phi_{u_\ve}(t,\cdot)$ in terms of $\ve>0$ for any $t\in [0,T_m^\ve)$ for which $u_\ve(s)\in S_2^{2+2\alpha}(\kappa)$ for $0\le s\le t$ (in the notation of Section~\ref{Sec5.1}). For such $t\in [0,T_m^\ve)$, one deduces the crucial trace estimate
$$
\|\partial_\eta\Phi_{u_\ve} (t,\cdot,1)\|_{W_2^{1/2}(I)}\le K\,\ve\ ,
$$
where $K>0$ depends on $\kappa\in (0,1)$ but not on $\ve>0$. On the one hand, this entails an $H^{2\alpha}(I)$-estimate on the nonlinearity $g_\ve(u_\ve(t))$ independent of $\ve$ and thus implies that there are $T_0>0$ and $\ve_0>0$ such that $u_\ve(t)$ stays in a set  $S_2^{2+2\alpha}(\kappa/2)$ for $t\in [0,T_0]$, uniformly with respect to $\ve\in (0,\ve_0)$. This in turn yields compactness of $(u_\ve)_{\ve\in (0,\ve_0)}$ and hence a cluster point $u_0$ thereof. On the other hand, the trace estimate on $\partial_\eta\Phi_{u_\ve}$  implies the convergence of the nonlinearity
$$
g_\ve(u_\ve) \longrightarrow \frac{1}{(1+u_0)^2} \ \ \text{ as $\ve\longrightarrow 0$}
$$
which shows that the cluster point $u_0$ is indeed a solution to the vanishing aspect ratio equation \eqref{uSG}. Summarizing, one obtains:

%%%%%%%%%%%%%%%%%%%%%%%
\begin{theorem}[Vanishing Aspect Ratio Limit \cite{ELW14, LW_v4}]
Let either $\beta>0$, $\gamma\ge 0$, $\tau\ge 0$ or $\beta=\gamma = 0$, $\tau>0$. For $\varepsilon>0$ let $(u_\varepsilon,\psi_\varepsilon)$ be the unique solution to \eqref{rapu1d}-\eqref{rapicu1d} on the maximal interval of existence $[0,T_m^\varepsilon)$ as provided by Theorem~\ref{T6}. There is $T_0>0$ such that
\begin{equation}\label{z01} 
u_{\varepsilon}\longrightarrow u_0\quad \text{in}\quad C\big([0,T_0],H^{1}(I)\big)
\end{equation}
and 
\begin{equation}\label{z00} 
\lim_{\varepsilon\to 0} \left[ \int_{\Omega(u_{\varepsilon}(t))} \left| \psi_{u_{\varepsilon}}(t,x,z) - \frac{1+z}{1+u_{\varepsilon}(t,x)} \right|^2 \rd(x,z) + \left\| g_{\varepsilon}(u_{\varepsilon}(t)) - \frac{1}{(1+u_{\varepsilon}(t))^2} \right\|_2 \right]= 0
\end{equation}
for all $t\in [0,T_0]$, where $u_0$ is the solution to the vanishing aspect ratio equation \eqref{uSG}. If $\lambda>0$ is sufficiently small, then $T_{m}^\ve=\infty$ for any $\ve$ small, and the statements \eqref{z01}-\eqref{z00} are true for each $T_0>0$.
\end{theorem}
%%%%%%%%%%%%%%%%%%%%%%%

The norm convergence stated in \eqref{z01} can be improved, in particular in the case $\beta>0$ (and also when $\gamma>0$).

%%%%%%%%%%%%%%%%%%%%%%%
%%%%%%%%%%%%%%%%%%%%%%%
\subsection{Dynamics in the Vanishing Aspect Ratio Model}\label{Sec5.5}
%%%%%%%%%%%%%%%%%%%%%%%
%%%%%%%%%%%%%%%%%%%%%%%

As for the stationary vanishing aspect ratio model, we consider a more general geometry and let $D$ be a bounded domain of $\R^d$, $d\ge 1$, with sufficiently smooth boundary. The dynamics of the vanishing aspect ratio model 
\begin{equation}
\begin{split}
\gamma^2 \partial_t^2 u + \partial_t u + \beta \Delta^2 u - \tau \Delta u & = - \frac{\lambda}{(1+u)^2}\ , \quad t>0\ , \quad x\in D\ , \\
u &= \beta \partial_\nu u  = 0 \ , \quad t>0\ , \quad x\in \partial D\ , \\
u(0,x) = \gamma \partial_t u(0,x) &= (0,0)\ , \quad x\in D\ ,
\end{split} \label{Mozart}
\end{equation}
is better understood, not only because it is simpler than the free boundary problem \eqref{rapu1d}-\eqref{rapicu1d} but also as it is a model semilinear evolution equation with a singular reaction term. The study of the latter was actually motivated by the possible occurrence of the quenching phenomenon, which corresponds to a finite time blowup of the  reaction term as the solution attains singular values of the nonlinearity, see \cite{FHQ92, Lev89, Lev92} for instance. In fact, several techniques developed for that purpose have been subsequently applied to \eqref{Mozart} and its variants. Obviously, the behavior of solutions to \eqref{Mozart} depends upon whether or not the parameters $\beta$ and $\gamma$ are positive. Before briefly reviewing the available results it is worth pointing out that, in contrast to the free boundary problem \eqref{rapu1d}-\eqref{rapicu1d}, local existence in a suitable functional setting of a solution to \eqref{Mozart} up to a maximal existence time $T_m>0$ readily follows from the properties of the linear part of the equation by a classical fixed point argument. In particular, given $T>0$, one has
\begin{equation}
T_m\ge T \quad\text{ whenever }\quad \inf_{(t,x)\in ([0,T]\cap [0,T_m)) \times D}\{u(t,x)\} > -1\ , \label{Buxtehude}
\end{equation}
which provides an explicit criterion for global existence.

Let us now be more specific and begin with the parabolic second-order case 
\begin{equation}
\begin{split}
\partial_t u - \tau \Delta u & = - \frac{\lambda}{(1+u)^2}\ , \quad t>0\ , \quad x\in D\ , \\
u & = 0 \ , \quad t>0\ , \quad x\in \partial D\ , \\
u(0,x) & = 0\ , \quad x\in D\ ,
\end{split} \label{Pachelbel}
\end{equation}
corresponding to $\gamma=\beta=0$ in \eqref{Mozart}. In that case, the comparison principle represents a powerful tool. In particular, recalling that $\lambda_0^{stat}$ is the threshold value of the parameter $\lambda$ above which no stationary solution to \eqref{Pachelbel} exists and below which there is a stable stationary solution $U_\lambda$, see Section~\ref{Sec4.2}, we have the following results: if $\lambda<\lambda_0^{stat}$, then the negativity of $U_\lambda$ and the comparison principle entail that $U_\lambda\le u(t)\le 0$ for all $t\in [0,T_m)$. Since $\min\{U_\lambda\}>-1$, we readily deduce from \eqref{Buxtehude} that $T_m=\infty$, see \cite{FMPS06, GhG08b}. This result actually implies that the pull-in voltages defined in Section~\ref{Sec1.3} satisfy $\lambda_0^{stat} \le \lambda_{0,0}^{dyn}$. Next, when $\lambda>\lambda_0^{stat}$, touchdown necessarily occurs and $T_m<\infty$ as shown in \cite{FMPS06, GhG08b, GPW05} with the help of a technique originally introduced in \cite{Lac83}. This entails that $\lambda_{0,0}^{dyn} \le \lambda_0^{stat}$ and thus that they are equal as announced in \eqref{lambda*}. The case $\lambda=\lambda_0^{stat}$ is also studied in \cite{EGG10, Guo08b} while additional information on the behavior near the touchdown time is provided in \cite{FMPS06, Guo08a, GPW05}. In particular, if $D=\mathbb{B}_1$ is the unit ball of $\R^d$ and $T_m<\infty$, then touchdown necessarily occurs at $x=0$ and only there. 

\medskip

As soon as either $\beta$ or $\gamma$ is not equal to zero, the comparison principle is no longer available for the evolution problem. This leads to a dramatic change in the analysis. Still, taking advantage of the existence of the branch of stable stationary solutions to \eqref{Mozart} described in Section~\ref{Sec4.2} when either $\beta=0$ or $\beta>0$ and $D=\mathbb{B}_1$, a technique developed in \cite{Lac83} allows one to prove that touchdown occurs as soon as $\lambda>\lambda_0^{stat}$, see \cite{KLNT11, KLNT15, LLZ14} when $\beta=0$ and \cite{LW14b} when $\beta>0$ and $D=\mathbb{B}_1$. Consequently $\lambda_{0,\gamma}^{dyn} \le \lambda_0^{stat}$, an inequality which is compatible with \eqref{lambda**} and  valid for all $\gamma\ge 0$. As for global existence it is shown in the aforementioned references that global solutions exist provided $\lambda$ is sufficiently small, without any restriction on the domain $D$. But no precise quantitative bound on the threshold value of $\lambda$ is known. We also refer to \cite{LiL12, LLS13} for numerical simulations and detailed descriptions of the touchdown behavior by asymptotic methods for \eqref{Mozart} when $\beta>0$ and $\gamma=0$. In particular, there is numerical evidence that, if $d=1$, touchdown takes place simultaneously at two different points for $\lambda$ large enough, a situation which contrasts markedly with the second-order case where such a behavior is only observed with a non-constant permittivity profile \cite{Guo08a}. Similar global existence results for small values of $\lambda$ are true for the vanishing aspect ratio model with $\beta>0$ and $\gamma>0$ when the clamped boundary conditions are replaced by pinned boundary conditions $u=\beta\Delta u = 0$ on $\partial D$\cite{Guo10}. In fact, considering pinned boundary conditions does not help much  for the evolutionary problem as far as the question of global existence is concerned since the comparison principle is not available in both cases, in contrast to the stationary setting described in Section~\ref{Sec4.2}. Finally, the behavior as $\gamma\to 0$ is analyzed in \cite{LiL16, LLZ14} when $\beta=0$.

To end this overview of results pertaining to the time-dependent vanishing aspect ratio model for MEMS we shortly list some contributions investigating stationary and time-dependent extensions of \eqref{Mozart}, some of them having been introduced in Section~\ref{Sec2.4}: 
\begin{itemize}
\item {\it Non-constant permittivity profiles}: It is already observed in \cite{GPW05, Pel01a} that a non-constant permittivity profile (corresponding to $\chi=0$ and a non-constant $f$ in \eqref{sccs}) has a strong influence on the threshold value of $\lambda$ as well as on the touchdown phenomenon. Contributions to this issue may be found in \cite{GhG06, GhG08a, GhG08b, GuS15, Guo08a, Guo08b, EGG10, KMS08, Wan15} for the second-order case ($\beta=0$, $\tau>0$) and in \cite{COG09} for the fourth-order equation ($\beta>0$).
\item {\it Capacitance}: The resulting equations are the object of several research papers and we refer to \cite{LiZ14, LiZ15, LiY07, GHW09, GuH14, GuK12, Hui11a, Hui11b, KLNT11, KLN16, PeC03, PeT01, Miy15} when $f\equiv 1$ and $\chi>0$ in \eqref{sccs}. The case of a non-constant permittivity profile $f$ with $\chi>0$ is investigated in \cite{CFT11, CFT14}.
\item {\it Fringing fields}: The influence of fringing fields \eqref{frfi} has also attracted attention recently and has been analyzed in particular in \cite{PeD05, DaW12, LiW12, LuYxx, WeY10, Wan13, FuW15, PaX15}.
\item {\it Van der Waals} or {\it Casimir forces}: The inclusion of \eqref{vdWCa} is studied in \cite{Lai15}.
\item When $\beta=0$ nonlinear effects due to large deformations have also been accounted for in the vanishing aspect ratio model \eqref{uSG} by replacing the Laplace operator $\Delta u$ by the mean curvature operator $\mathrm{div}\left( \nabla u / \sqrt{1+\varepsilon^2|\nabla u|^2} \right)$. So far only the stationary problem has been studied, the outcome being that the nonlinear diffusion alters significantly the structure of stationary solutions \cite{BrP11, BrP12, BrL13, CHW13, PaX12, PaX15}.
\item Control and inverse problems for MEMS using the vanishing aspect ratio model \eqref{uSG} are investigated in \cite{CKL09, ClK14}.
\item Finally, the extension to $(1+u)^{-p}$ for $p>1$ of the above mentioned results has been studied and shown to depend strongly upon both $p$ and $d$, see \cite{Hui09, GLY14, CEG10, CoG10, CHW13, PaX15, ZhL12, WeY10, Miy15, DWW16} and the references therein. 
\end{itemize}

%%%%%%%%%%%%%%%%%%%%%%%
%%%%%%%%%%%%%%%%%%%%%%%
\section{The Three-Dimensional Problem}\label{S6}
%%%%%%%%%%%%%%%%%%%%%%%
%%%%%%%%%%%%%%%%%%%%%%%

The results presented in Section~\ref{S4} and Section~\ref{S5} concern the free boundary problem  \eqref{psi1d}-\eqref{icu1d} with a two-dimensional region $\Omega(u)$. As noted in Section~\ref{Sec2.2} this is a special case of \eqref{psi}-\eqref{icu}
when $D$ is assumed to be a rectangle and  zero variation in the $y$-direction is presupposed. Of course, a more general (two-dimensional) geometry of $D$ is of great interest with respect to real world applications. But with a general two-dimensional domain $D$, the resulting $\Omega(u)$ is a three-dimensional Lipschitz domain and it is not clear that  the trace of $\nabla\psi$ is well-defined on the boundary $\partial\Omega(u)$ thereof. Nevertheless, if $D$  is assumed to be a bounded and {\it convex} domain in  $\R^2$ with a sufficiently smooth boundary, a similar, but technically advanced argumentation as outlined in Section~\ref{Sec3.1} (see \cite{LW_3d} for details) entails that solutions to the elliptic problem \eqref{psi}-\eqref{bcpsi} are sufficiently smooth so that  the trace of $\nabla\psi$ on $\partial\Omega(u)$ has a perfect meaning. However, the regularity of the square of the gradient trace of such solutions is weaker than before and the two-dimensional analogue of the function $g_\ve$ from Proposition~\ref{P1} is less regular than stated therein in the sense that it maps $S_q^2(\kappa)$ with $q>3$ only into $L_2(D)$ (instead of some Sobolev space with positive order). As a consequence, the results presently known for the three-dimensional problem are restricted to the fourth-order case $\beta>0$. 

So, let $\beta>0$, $\gamma=0$, and consider a bounded and convex domain $D$  in  $\R^2$ with a smooth boundary. Then one can prove a local existence result for the dynamic problem \eqref{psi}-\eqref{icu} with arbitrary $\lambda>0$ as in Theorem~\ref{T6}. Also Theorem~\ref{T11} on global existence of solutions for small $\lambda$ values is still valid. As for stationary solutions, the same statement on existence of such solutions for small $\lambda$ as in Theorem~\ref{T1} is true except for the negativity of $U_\lambda$. Nevertheless, this last property holds true  if $D$ is a disc so that the maximum principle for radially symmetric functions is available according to Proposition~\ref{P4}. In this case, the stationary solutions provided by the analogue of Theorem~\ref{T1} are radially symmetric and the non-existence result of Theorem~\ref{T2} is valid as well. For details on all this we refer to \cite{LW_3d}.

 Let us finally mention that the implicit function theorem is also used in \cite{Cim07} to study stationary solutions in the second-order case $\beta=0$ for small values of $\lambda$.

%%%%%%%%%%%%%%%%%%%%%%%
%%%%%%%%%%%%%%%%%%%%%%%
\section{Summary and Open Problems}\label{S7}
%%%%%%%%%%%%%%%%%%%%%%%
%%%%%%%%%%%%%%%%%%%%%%%

In the previous sections we reviewed the derivation of the free boundary problem \eqref{psi}-\eqref{icu} as well as the mathematical results obtained so far for its two-dimensional version \eqref{psi1d}-\eqref{icu1d}. We have in particular shown that it is locally well-posed when either $\beta>0$ and $\gamma\ge 0$ or $\beta=\gamma=0$ and $\tau>0$, having also global solutions and stable stationary solutions for small values of $\lambda$ in both cases (Theorem~\ref{T6}, Theorem~\ref{T11}, and Theorem~\ref{T1}). These results remain true for the nonlinear elasticity model \eqref{uquasilin}, see \cite{ELW14, LW14c}, as well as when self-stretching forces are taken into account in addition to bending, that is, $\beta>0$ and the second-order contribution $-\tau\partial_x^2u$ in \eqref{u1d} is replaced by
\begin{equation}
- \big(\tau+a\|\partial_x u\|_2^2\big)\partial_x^2 u\ , \qquad a>0\ . \label{self}
\end{equation}
The local and global existence results also extend to initial conditions which are not identically zero under suitable smoothness assumptions, possibly supplemented with smallness conditions  for global existence. We have also proved that there is no stationary solution for large values of $\lambda$ (Theorem~\ref{T2}) while, when $\gamma=\beta=0$, a finite time singularity occurs for the evolution problem (Theorem~\ref{T12}). This also holds true for the nonlinear elasticity model \eqref{uquasilin} with the same range of parameters $\gamma=\beta=0$ \cite{ELW13, ELW14} and extends as well to a broader class of initial data. We summarize the results obtained so far for \eqref{psi1d}-\eqref{icu1d} and its stationary counterpart in the table below, assuming that $\tau>0$.

\bigskip

\begin{tabular}{|c|p{2.5cm}|p{3cm}|p{2.5cm}|p{2.5cm}|p{2.5cm}|}
\hline
\eqref{psi1d}-\eqref{icu1d} & \multicolumn{3}{c|}{Evolution Problem} & \multicolumn{2}{c|}{Steady States} \\
 \hline
 & local existence & global existence for small $\lambda$ & finite time singularity & non-existence for large $\lambda$ & multiplicity for small $\lambda$\\
\hline \hline
$\gamma=\beta=0$ & true & true & true & true & unknown \\
\hline
$\gamma>0$, $\beta=0$ & unknown & unknown & unknown & true & unknown \\
\hline
$\gamma=0$, $\beta>0$ & true & true & unknown & true & true \\
\hline
$\gamma>0$, $\beta>0$ & true & true & unknown & true & true \\
\hline
\end{tabular}

\bigskip

Though the above table contains quite a number of affirmative answers to the basic questions we have investigated recently, a few of them are, however, still not answered in a satisfactory or complete way. Below we list some open problems which we believe to be of importance for a better understanding of the free boundary model \eqref{psi1d}-\eqref{icu1d}:
\begin{enumerate}
\item The existence of the threshold value $\lambda_\varepsilon^{stat}$ is yet unknown. 
\item By Theorem~\ref{T3} there are countably many values of $\lambda$ in a neighborhood of zero for which there are at least two stationary solutions to \eqref{psi1d}-\eqref{icu1d} when $\beta>0$. That this is true for all values of $\lambda$ in a neighborhood of zero would be a first step towards a finer description of the structure of stationary solutions. We actually conjecture that the bifurcation diagram in that case is similar to the one of the vanishing aspect ratio model stated in Theorem~\ref{TSSIntroduction}. No clue towards a proof of a similar result when $\beta=0$ is currently available.
\item According to Theorem~\ref{T12} a finite time singularity occurs for all solutions to the evolutionary problem \eqref{psi1d}-\eqref{icu1d} when $\beta=\gamma=0$, $\tau>0$, and $\lambda$ is large enough. It is, however, unknown whether touchdown (in the sense of \eqref{PLtouchdown}) indeed takes place.
\item When $\beta>0$, it is yet unproved that a finite time singularity may occur for $\lambda$ large enough.
\item If touchdown occurs at a finite time $T_m$, nothing is known yet about the structure of the touchdown set $\mathcal{S} := \{ x \in (-1,1)\ :\ u(T_m,x)=-1\}$ and the spatio-temporal scales of the evolution of $(u,\psi_u)$ as $(t,x)$ approaches $(T_m,x_0)$ for $x_0\in\mathcal{S}$. 
\item When $\beta>0$, self-stretching forces corresponding to \eqref{self} received little attention in the literature, even for the vanishing aspect ratio model, and their influence is far from being well understood.
\end{enumerate}

%%%%%%%%%%%%%%%%%%%%%%%
%%%%%%%%%%%%%%%%%%%%%%%
%\nocite{*}
%
\bibliographystyle{siam} 
\bibliography{MEMS}
%%%%%%%%%%%%%%%%%%%%%%%
%%%%%%%%%%%%%%%%%%%%%%%

%%%%%%%%%%%%%%%%%%%%%%%
%%%%%%%%%%%%%%%%%%%%%%%
%%%%%%%%%%%%%%%%%%%%%%%
%%%%%%%%%%%%%%%%%%%%%%%
\end{document}